\documentclass[a4paper]{amsart}
\usepackage[theoremdefs]{allergy}

\usepackage{enumitem}

\usepackage{tikz}
\usetikzlibrary{decorations.pathreplacing,shapes,arrows,positioning}
\tikzset{
	vertex/.style = {
		circle,
		fill            = black,
		outer sep = 2pt,
		inner sep = 1pt,
	}
}
\usetikzlibrary{spy}
\usetikzlibrary{matrix}

\usepackage{pgfplots}

\usepackage{siunitx} 

\usepackage{pdfsync}
\graphicspath{{fig/}}

\usepackage[labelformat=simple]{subcaption}

\usepackage{pdfsync}
\usepackage{bbm}

\newcommand*\RR{\mathbb{R}}
\newcommand*\Z{\mathbb{Z}}
\newcommand*\ee{\mathrm{e}}
\newcommand*\ii{\mathrm{i}}

\newcommand*\imcone{A(\MM_+)}

\newcommand {\MM} { {\mathcal{M}} }

\newcommand*\loss{\ell}

\renewcommand{\geq}{\geqslant}
\renewcommand{\leq}{\leqslant}

\renewcommand{\geq}{\geqslant}
\renewcommand{\leq}{\leqslant}

\DeclarePairedDelimiter\ceil{\lceil}{\rceil}

\newcommand {\Chi} {{\bf \raise 2pt \hbox{$\chi$}} }
\newcommand*\ndet{m}

\newcommand{\scl}[2]{\pairing*{#1}{#2}}

\newcommand {\R} { {\mathbb R} }

\newcommand*\dd{\mathrm{d}}

\newcommand{\supp}{\operatorname{supp}}

\newcommand{\beq}{\begin{equation}}
\newcommand{\eeq}{\end{equation}}
\newcommand{\bea} {\begin{array}{rl}}
\newcommand{\eea} {\end{array}}
\newcommand{\bepa}{\left\{ \begin{array}{l}}
\newcommand{\eepa} {\end{array}\right.}
\newcommand{\bmu}{\begin{multline}}
\newcommand{\emu}{\end{multline}}
\DeclareMathOperator*{\argmin}{arg\,min}

\title{Linear inverse problems with nonnegativity constraints: singularity of optimisers}

\date{\today}
\author{Camille Pouchol}
\address{Universit\'e Paris Cit\'e, FP2M, CNRS FR 2036, MAP5 UMR 8145, F-75006 Paris, France.}
\email{camille.pouchol@u-paris.fr}

\author{Olivier Verdier}
\address{Department of Computing, Electrical Engineering and Mathematical Sciences, Western Norway University of Applied Sciences, Bergen, Norway.}
\email{olivier.verdier@hvl.no}

\begin{document}

\newcounter{assum}
\begin{abstract}
  We look at continuum solutions in optimisation problems associated to linear inverse problems $y = Ax$ with non-negativity constraint $x \geq 0$.
  We focus on the case where the noise model leads to maximum likelihood estimation through general  divergences, which cover a wide range of common noise statistics such as Gaussian and Poisson.
  Considering~$x$ as a Radon measure over the domain on which the reconstruction is taking place,  we show a general singularity result.
  In the high noise regime corresponding to $y \notin \setc{Ax}{x \geq 0}$ and under a key assumption on the divergence as well as on the operator $A$, any optimiser has a singular part with respect to the Lebesgue measure.
  We hence provide an explanation as to why any possible algorithm successfully solving the optimisation problem will lead to undesirably spiky-looking images when the image resolution gets finer, a phenomenon well documented in the literature.
  We illustrate these results with several numerical examples inspired by medical imaging.
\end{abstract}

\maketitle
\section{Introduction}
We consider linear inverse problems of the type $Ax = y$, where \(A\) is a linear operator, and \(y\) lies in a suitable linear space.
These problems are often endowed with constraints stemming from the model at hand.
One of the most common such constraints is nonnegativity of the unknown, {i.e.}, $x \geq 0$.
This happens in various applications, in particular in medical imaging, where \(x\) is the \emph{activity}, which is nonnegative for physical reasons.
One important example is that of Positron Emission Tomography (PET)%
~\cite{Shepp1982}.
In this setting, the unknown \(x\) lies in $\R^r$, the data $y$ lies in  $\R^\ndet$, and the operator $A$ lies in $\R^{\ndet \times r}$, where $\ndet$ is the number of data points, $r$ is the number of voxels. 
Deconvolution problems often also incorporate such constraints~\cite{Henrot2012}.

Depending on the noise model, the corresponding (negative) log-likelihood problem is typically of the form
\[ \min_{x \geq 0} \quad D(y,Ax),\]
where $D$ is some divergence functional.
If the noise model is Gaussian, for instance, then $D$ is simply the Euclidean distance, whereas if the noise model is Poisson, $D$ is the Kullback--Leibler divergence. 

We analyse the effect of an ever increasing resolution, which leads us to regard the unknown $x$ as a function (henceforth denoted $\mu$) in some functional space $X$ of functions over a compact set $K\subset \R^p$, and the operator $A$ is now a linear mapping from the space $X$ to $\R^\ndet$.
This leads to the optimisation problem 
\begin{equation}
\label{opt_int}
 \min_{\mu \geq 0} \quad D(y,A\mu).
 \end{equation}

 The non-negativity constraint has been shown to cause sparsity in various contexts in optimisation and optimal control~\cite{Clason2019}.
 As a result, the proper functional space $X$ to be considered appears to be that of Radon measures, where discrete measures are regarded as sparse, with the operator $A$ defined for all $\mu \in X, \mu \geq 0$ by
 \[ (A \mu)_i := \int a_i \, \dd\mu, \qquad a_i \geq 0.\]
 
 Here, the assumption $a_i \geq 0$ reflects our interest in applications where the measured data is nonnegative. 
 
 This sparsity phenomenon has been observed in optimal control~\cite{Clason2017, Loheac2017}, as well as for the optimisation problem~\eqref{opt_int} when the divergence $D$ is specifically the Kullback--Leibler divergence~\cite{Verdier2020}.
 This divergence and the framework of Radon measures can even be natural from the physical point of view, as is the case in PET where the underlying model is based on Poisson point processes~\cite{Shepp1982, Verdier2021}.

 In the examples above, sparsity (which can arise in the form of Dirac masses) is undesirable as the sought-for image is expected to be at least piecewise smooth.
 In other contexts, sparsity of the signal must be enforced, as is the case for instance in sparse super-resolution where Dirac masses are the desired outcome~\cite{Denoyelle2017, Debarre2022}. 

In the latter case, the situation is completely different as the unknown signal is \emph{known} to be a sum of (nonnegative) Dirac masses and the aim is to make sure that the optimisation problem has \emph{at least one minimiser} which is a sum of Dirac masses, thereby hoping to recover the number and support of the unknown discrete measure. A very general treatment of the existence of such sparse optimisers has recently been given in~\cite{Boyer2019, Bredies2020}.


The goal of the present work is to establish general conditions for the operator~$A$ and the divergence $D$ to fulfill, so that, when the measured data \(y\) is not in the cone \(\setc{A \mu}{\mu \geq 0}\), \emph{all minimisers} are singular in the sense defined below. We use Lebesgue's decomposition theorem with respect to the Lebesgue measure: any Radon measure $\mu$ over $K$ decomposes into the sum of an absolutely continuous measure (its \textit{absolutely continuous part}) and a singular measure  (its \textit{singular part}).
\begin{definition}
\label{def:sparse} A measure $\mu$ will be said to be \textit{singular} if its singular part is non-zero, and \textit{completely singular} if it is singular and its absolutely continuous part is zero. 
\end{definition}
The paper~\cite{Verdier2020} established that optimal measures are completely singular in the specific case where
\begin{enumerate}[label={(\roman*)}, ref={(\roman*)}]
\item the functions $a_i$ are linearly independent and (real) analytic,
  \label{it:linindep}
\item $D$ is the Kullback-Leibler divergence.
  \label{it:kldiv}
\end{enumerate}
The first assumption \ref{it:linindep} is restrictive since many practical problems are such that all the functions $a_i$ are compactly supported inside the interior of $K$. This makes analyticity irreconcilable with the condition $a_i \geq 0$. 
The second assumption \ref{it:kldiv} also restricts the possible noise models to the single, albeit important, case of Poisson measurements.  

In fact, the so-called \emph{$\beta$-divergences}  have attracted interest recently in non-negative matrix factorisation~\cite{Fevotte2011}, and are  now also advocated for in some medical imaging contexts, such as in PET~\cite{Cavalcanti2019}.
This is a family of divergences depending on a real parameter \(\beta\).
It has the appealing property of interpolating between the Kullback--Leibler divergence ($\beta = 1$) and the Euclidean distance ($\beta = 2$)~\cite{Cichocki2010}, which correspond to different noise models (see \autoref{sec:noisemodels}).


In this article, we first generalise the above result (when~\ref{it:linindep} holds) to very general divergences.

Then, we proceed to treat the case of more general detector functions $a_i$ relaxing assumption~\ref{it:linindep}. This forces us to restrict the class of divergences.
Informally at this stage, we hence consider
\begin{itemize}
\item[(C)]  detector functions $a_i$ which are \emph{locally} linearly independent and \emph{piecewise} analytic,
\item[(H)] divergences which are sufficiently averse to zero values whenever the data is non-zero (see the precise statement in equation \eqref{H} in \autoref{sec:divergences}).
\end{itemize}
Most functions we are aware of are expected to satisfy $(C)$, such as those defining PET (although local linear independence might be difficult to prove for each particular case). Condition $(H)$ limits the class of divergences, as it for instance excludes the Euclidean distance. It still covers a wide class of $\beta$-divergences, f-divergences and Bregman divergences, in particular those which arise in contexts featuring nonnegativity constraints.

Our main result may then be stated as follows (see \autoref{sparse} for a precise statement).
\begin{theorem}
\label{thm1}
If $y \notin \{A \mu, \; \mu \geq 0\}$ with $A$ satisfying $(C)$ and $D$ satisfying $(H)$, then all optimal solutions $\mu^\star$ to~\eqref{opt_int} are singular.
\end{theorem}
In other words, if the data is not in the image of the cone of nonnegative measures $\{\mu \geq 0\}$ under the operator~$A$, any optimiser $\mu^\star$ is singular, even if the data has been generated from an image that is absolutely continuous with respect to the Lebesgue measure.

The condition that the data \(y\) does not belong to the image cone \(\setc{A \mu}{\mu \geq 0}\)  can be interpreted as a condition on the level of noise: the more noise there is, the more likely it is that this condition be fulfilled (see~\cite{Verdier2020} and \autoref{sec:noisemodels}).
Undesirable singular-looking images will hence arise in the high-noise regime.

Our results show that the optimisation problem itself leads to singular results.
Thus, any algorithm successfully solving~\eqref{opt_int} will inevitably lead to spiky-looking images as one keeps iterating.
In the context of medical imaging, this has been observed when using the Maximum-Likelihood-Expectation-Maximisation (ML-EM, also called the Richardson--Lucy algorithm) for solving~\eqref{opt_int}, and has been referred to in the literature as the ``night-sky'' or the ``draughtsboard'' effect~\cite{Vardi1985}.

Consequently, the same kind of artefacts will be observed for other likelihoods, hinting at the necessity of either early stopping when solving~\eqref{opt_int} (see \cite{Resmerita2007}, \cite{Oktem2019}), or adding appropriate regularisation terms of the form 
\begin{equation}
\label{reg}
 \min_{\mu \geq 0} \quad D(y,A\mu) + \lambda R(\mu),
 \end{equation}
with regularisation parameter $\lambda$ that is sufficiently large to alleviate the issue~\cite{Cavalcanti2019}.

In this work, we also provide numerical examples coming from PET where, when solving~\eqref{opt_int} for some commonly used divergences, and applying a sufficient amount of noise, reconstructions exhibit the night-sky effect (\autoref{sec5}).
As our theoretical results suggest, this effect should be more and more prominent as one keeps iterating a convergent algorithm for solving~\eqref{opt_int}, or a convergent algorithm for solving~\eqref{reg} with a sufficiently small regularisation parameter~$\lambda$.

\vspace{0.5cm}
\subsection*{Outline of the paper}
The paper is organised as follows.
In \autoref{sec2}, we define the inverse problem by setting the functional analytic framework as well as the family of divergences being considered, leading to the corresponding maximum likelihood problem.
In \autoref{sec3}, we proceed to study the resulting optimisation problem and prove our main result~\autoref{thm1}.
Section \autoref{sec4} is devoted to discussing the assumption~$(C)$ on the detector functions in more detail, in the case of a toy example and for the $2$-dimensional PET operator on a regular polygon.
We finally illustrate our results about singularity by numerical simulations in \autoref{sec5}.
They feature different examples with and without regularisation.

\section{Inverse Problem Setup}
 \label{sec2}
 \subsection{Linear inverse problem}
 \label{sec:lininv}
  We aim at reconstructing an image $\mu$ defined on a non-empty compact $K \subset \R^p$, $p \geq 1$, where $K$ is the closure of a bounded Lipschitz connected open set.
 \subsubsection{Measure-theoretic background}
The unknown image \(\mu\) is an element of the space of Radon measures, denoted $\MM(K)$, which is the topological dual space of continuous functions over the compact, denoted $\mathcal{C}(K)$.
 We endow $\MM(K)$ with the weak-$\ast$ topology, making $\mathcal{C}(K)$ its dual space.
 We also recall that the bounded sets of $\MM(K)$ are relatively compact in the weak-$\ast$ topology, by virtue of the Banach--Alaoglu Theorem~\cite{Brezis2010}.
 
 We denote the dual pairing between a function $\mu \in \MM(K)$ and a function $f \in \mathcal{C}(K)$ 
 by $\scl{\mu}{f}$, and $\MM_+(K)$ stands for the set of nonnegative Radon measures. Whenever the context is clear, we shall drop the reference to the compact $K$ and write $\MM$, $\MM_+$ and $\mathcal{C}$.
 
Finally, we recall that by Lebesgue's decomposition theorem, any measure
$\mu \in \MM_+(K)$ can uniquely be written $\mu = \mu_1 + \mu_2$, where $\mu_1 \in \MM_+(K)$ is absolutely continuous, and $\mu_2 \in \MM_+(K)$ is singular, where absolute continuity and singularity are meant with respect to the Lebesgue measure. The measure $\mu$ is then said to have a \textit{singular part} if $\mu_2 \neq 0$.
 
  \subsubsection{Operator}

 A data point \(y\) is a vector of $\ndet$ scalar nonnegative measurements, that is, 
 \[y \in \R_+^\ndet.\] 
 This vector itself typically is the realisation of some random variable with mean~$A \mu$, where $\mu \in \MM_+(K)$ is the image to be reconstructed, and $A$ is a linear mapping $A \colon \MM(K) \to \R^\ndet$.

The only assumption we make on $A$ is that it is continuous in the weak-$\ast$ topology.
From~\cite[Proposition 3.14]{Brezis2010}, this implies that $A$ is of the form 
\begin{equation}
\label{scalar} (A \mu)_i = \scl{\mu}{a_i}, \quad i = 1, \ldots, \ndet,
\end{equation} where the \emph{detector functions} $a_i$ are elements of $\mathcal{C}(K)$.
This covers the case of PET~\cite{Mair1996, Verdier2020} and more generally the setting of kernel operators: if the underlying operator in infinite dimension is of the form 
\[ \mu \longmapsto  \int_K k(\cdot,y) \, \dd\mu(y),\] for some smooth kernel $k \in \mathcal{C}(K \times K)$, the operator $A$ typically is obtained from a sampling for $\ndet$ points $x_i \in K$ or integrating the kernel over some subdomains $\Omega_i \subset K$, namely 
\[a_i = k(x_i, \cdot),  \quad \text{ or }\quad  a_i = \int_{\Omega_i} k(x, \cdot) \, \dd x, \quad i = 1, \ldots, \ndet.\]
We will make the assumption that $A$ maps $\MM_+(K)$ into the set of (componentwise) nonnegative vectors denoted $\R_+^\ndet$ which of course is equivalent to 
\begin{equation}
\label{nonnegative} 
a_i \geq 0, \qquad  i = 1, \ldots, \ndet.
\end{equation}
Finally, we assume $a_i \neq 0$ for all $i = 1, \ldots, \ndet$, as well as
\begin{equation}
\label{positive} 
\sum_{i=1}^\ndet a_i > 0.
\end{equation}
Indeed, at any point $x$ such that $\sum_{i=1}^\ndet  a_i(x) =0$, we would have no information on the unknown \(\mu\).

 Another consequence of the simple continuity assumption on $A$ is that its adjoint $A^\ast \colon \R^\ndet \to \mathcal{C}(K)$ is simply defined as 
 \[A^\ast \lambda =  \sum_{i=1}^\ndet \lambda_i a_i, \quad \lambda \in \R^\ndet.\]
 
 We shall sometimes need to know when $A^\ast$ is injective. This is of course equivalent to the linear independence of the family $(a_i)_{i = 1,\ldots, \ndet}$. 
 Note that since the codomain of \(A\) is finite dimensional, we have
 \[A^\ast \text{ is injective} \;  \iff \;  A  \text{ is surjective}.\]

 In order to solve the inverse problem with the nonnegativity constraint, we aim at solving the optimisation problem
 \begin{equation}
\label{opt}
 \operatorname*{min}_{\mu \in \MM_+(K)}  \; D (y ,A\mu).
 \end{equation}
Here, $D$ stands for some divergence over $\R_+^\ndet$, and we will use the notation 
\begin{equation}
\label{loss}
\loss(\mu) := D(y,A \mu).
\end{equation}
 
 Finally, let us define the notion of support for the various relevant cases (all these cases can be covered in one single definition, but we prefer separating them for clarity):
 \begin{itemize}
 \item for a vector $w \in \R_+^\ndet$, \[\supp(w) = \setc*{i \in \{1, \ldots, \ndet\}}{w_i > 0},\]
 \item for a nonnegative function $f \in \mathcal{C}(K)$,  \[\supp(f) = \overline{\setc*{x \in K}{f(x) > 0}},\]
 where $\overline{F}$ stands for the closure of a set $F \subset K$,
  \item for a nonnegative measure $\mu \in \MM_+(K)$,  \[\supp(\mu) := \setc[\big]{x \in K}{ \mu(N) > 0, \; \forall N \in N(x)},
\]
where $N(x)$ is the set of all open neighbourhoods of $x$. 
 \end{itemize}
 Notice that with these notations in place, we may now rewrite assumption~\eqref{positive} as \[\bigcup_{i=1}^\ndet \, \supp(a_i) = K.\]

 \subsection{Divergences}
 \label{sec:divergences}
We shall assume that $D$ is a \emph{separable} divergence, {i.e.}, 
\[ D(u,v) = \sum_{i=1}^\ndet d(u_i, v_i)
  \qquad
  u,\,v \in \R_+^\ndet
  .\]
where $d \colon \R_+ \times \R_+ \to \R_+ \cup \{+\infty\}$ is a (scalar) divergence.
When needed, we will implicitly extend the function~$d$ to the whole \(\RR^2\) by $+\infty$ for $u<0$ or $v<0$.
Throughout, we will assume that $d$ satisfies the following basic properties
  \begin{description}
\item[Separation]  \[d(u,v) = 0\quad \iff \quad  u = v.\]
\item[Convexity] for $u \in \R_+$, $v \mapsto d(u,v)$ is convex on $\R_+$.
\item[Regularity] for $u \in \R_+$, the mapping $v \mapsto d(u,v)$ is lower-semicontinuous on $\R_+$, 
\item[Coercivity] for $u \in \R_+$, the mapping $v \mapsto d(u,v)$ is coercive, {i.e.}, 
\[\forall u \in \R_+, \quad \lim_{v \rightarrow +\infty} d(u,v) = +\infty.\]
\end{description}

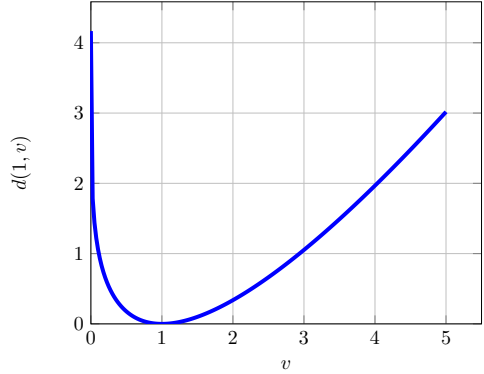
\begin{figure}[h]
  \centering
  \begin{tikzpicture}[scale=0.75]
    \begin{axis}[
      xlabel={\(v\)},
      ylabel={\(d(1,v)\)},
      grid=major,
      xmin=0,
      ymin=0,
      ]
     \addplot[blue, line width=2pt] table {divdata.dat};
    \end{axis}
  \end{tikzpicture}
  \caption{An example of map \(v \mapsto d(1,v)\) which fulfills the assumptions above.
    In particular, the derivative tends to \(-\infty\) when \(v\) approaches zero, and the divergence is zero only at the point \(v=1\).
Note that this is a plot of the function \(v \mapsto d_{\beta}(1,v)\) (see \autoref{sec:betdiv}) for \(\beta = 1.2\).
  }
\end{figure}
 
 The above assumptions ensure that, for a fixed $u \in \R_+$, $v \mapsto d(u,v)$ is subdifferentiable at every $v>0$. 
For our results about singularity of optimisers, we will also sometimes consider a subclass of such divergences for which we must specify (non-)subdifferentiability at $v = 0$ as follows: 
\begin{equation}
\label{H}
v \mapsto d(u,v) \text{ is not subdifferentiable at $0$ if $u>0$.}
\tag{H}
\end{equation}
From standard convex analysis, this is equivalent to the directional derivative of $d(u, \cdot)$ being $-\infty$ at $0$ for all $u>0$, {i.e.},
\begin{equation}
\eqref{H} \quad \iff \quad \forall u>0, \;  \lim_{t \rightarrow 0,  \, t>0} \frac{d(u,t) - d(u,0)}{t}
= -\infty.
\end{equation}

 Finally, we will need one very weak assumption involving both the operator $A$ and the divergence $D$:
 \begin{equation}
 \label{blanket}
 \text{there exists $\mu \in \MM_+$ such that $D(y, \cdot)$ is finite and continuous at $A \mu$}.
 \end{equation}
One sufficient and simpler assumption for the above condition to hold is the existence of some point $x \in K$ such that the divergences $d(y_i, \cdot)$ are continuous at $a_i(x)$ for all $i \in \{1, \ldots, \ndet\}$, which follows from choosing $\mu = \delta_{x}$.
 
  \subsection{Examples}
  \subsubsection{Beta Divergences}
  \label{sec:betdiv}

 For $u > 0$, $v>0$ scalar variables, $\beta \in (1,2]$, we define
\[ d_{\beta}(u, v) := \frac{1}{\beta(\beta - 1)}\left(u^\beta + (\beta - 1) v^\beta - \beta u v^{\beta - 1}\right),\]
which for $\beta = 2$ gives the Euclidean distance \[d_2(u,v) = \frac{1}{2}(u-v)^2,\]
and by continuity for $\beta = 1$ the Kullback--Leibler divergence
\[d_1(u,v) = u \ln\Big(\frac{u}{v}\Big) - u + v.\]
The corresponding divergence over $\R_+^\ndet$ will be denoted $D_\beta$.

More precisely, using the convention $0/0 = 0$, $0 \ln 0 = 0$, $d_{1}$ is defined for nonnegative scalars $u \geq 0$, $v\geq 0$ as follows:

  \[d_1(u,v) =
    \begin{cases}
      v & \text{if \(u = 0,\,v \geq 0\)},
      \\
+ \infty & \text{if \(u > 0,\, v = 0\)},
\\
 u \ln\big(\frac{u}{v}\big) - u + v & \text{if \(u>0,\,v>0\)}.
\end{cases}
\]
Note that the $\beta$-divergences $d_\beta$ satisfy Hypothesis~\eqref{H} for all $\beta \in [1,2)$, but $d_2$ does not.


 
\subsubsection{Reverse \(f\)- and Bregman divergences}
Given a convex function \(F \colon \RR_+ \to \RR_+ \cup \set{+\infty}\) with \(F(1) = 0\), one can define the following ``reverse \(f\)-divergence''
\[
  d^f_F(u,v) := F(v/u) u, \qquad u > 0 \quad v \geq 0,
\]
Similarly, with a convex function \(F \colon \RR \to \RR \cup \set{+\infty} \) such that  \(F\) is differentiable, we define the ``reverse Bregman divergence''
\[
  d^B_F(u,v) := F(v) - F(u) - F'(u)(v-u), \qquad u > 0 \quad v \geq 0.
\]

Both fulfill our assumptions as long as
\begin{itemize}
\item $F$ is lower semicontinuous,
\item \(F'(0) = -\infty\),
\item they are suitably extended to the case \(u=0\) in such a way that
  \(d(0,\cdot)\) is a nonnegative convex, lower semicontinuous and coercive function.
\end{itemize}

Note that $\beta$-divergences are not reverse f-divergences nor reverse Bregman divergences for $\beta \in (1,2)$, the Kullback--Leibler divergence $d_1$ is a reverse f-divergence but not a reverse Bregman divergence, and the Euclidean distance $d_2$ is a reverse Bregman divergence but not a reverse f-divergence.

   \subsection{The noise model.}
   \label{sec:noisemodels}
   In general, $y$ is drawn according to a distribution parameterised by $w = A \mu$ as well as some additional dispersion parameter $\phi>0$ controlling the noise level.
   For a better understanding of why the condition $y \notin  \imcone$ typically arises in the high noise case, we review the underlying statistical model in the specific case of $\beta$-divergences.

   In this setting, the minimisation problem~\eqref{opt} is (up to constants) the corresponding (negative-log) likelihood maximum problem.
   We write the statistical model for $y$ and $w$ as scalar variables, as the full statistical model is straightforwardly obtained by assuming independent components.

   A general way to write the noise model giving rise to $\beta$-divergences is to use the so-called Tweedie distributions~\cite{Simsekli2013}.
  The $\beta$-divergences are a special case of such distributions, as the corresponding Tweedie distribution is given for $w$ fixed by
\[ y \mapsto H_\beta(y,\phi) \operatorname{exp}\left(-\frac{1}{\phi} d_\beta(y,w)\right),\]
where $H_\beta$ is a normalisation factor for the above function to integrate to $1$.
Observing the data $y$ to estimate $w$, minimising the negative log-likelihood problem indeed is equivalent to minimising $w \mapsto d_\beta(y,w)$. We refer to~\cite{Simsekli2013} for more details.

The underlying density is not always tractable (this is the case if $1< \beta < 2$), making the noise model unclear. In other cases, the noise model can be further identified as follows.

\begin{description}
\item[Case $\beta = 2$] the noise model is Gaussian, {i.e.},  \[y \sim \mathcal{N}(w, \phi).\]
\item[Case $1< \beta<  2$] to the best of our knowledge, no explicit model is known.
\item[Case $\beta = 1$] the noise model is Poisson, {i.e.}, \[y \sim\phi \, \mathcal{P}\Big(\frac{1}{\phi} w\Big).\]
\end{description}

Furthermore, $y$ has mean $w$ and variance $\phi^{2-\beta} w$: $y$ concentrates around $w$ as the noise level $\phi$ vanishes.
Hence, the less noise there is, the more likely it is that $y \in \imcone$. For a more quantitative version of this statement in the case of $d = d_1$, {i.e.}, Poisson distributed measurements, see~\cite{Verdier2020}.
 
\begin{remark}
Note that the Gaussian case $\beta = 2$ is the only one which does not (necessarily) lead to nonnegative data $y \in \R_+^\ndet$. Yet, we will always make this assumption throughout. Indeed, the Gaussian noise model is frequently used as an approximate basic model, even for inverse problems where the data satisfies $y \geq 0$ (for physical reasons). The problem is then typically solved using nonnegative least-squares, hence it is worth studying this case as well.
\end{remark}



 \section{Optimisation problem}
 \label{sec3}
 We now investigate the optimisation problem~\eqref{opt}, starting with a related optimisation problem and its dual.
 \subsection{Related optimisation problem and its dual}
 We consider the cone
 \[
   \imcone = \setc*{A \mu}{\mu \in \MM_+}
   .
 \]
 We notice that \(\imcone\) is a closed convex set:

 \begin{lemma}
   \label{prop_cone_closed}
  The cone \(\imcone\) is closed.
 \end{lemma}
 \begin{proof}
   Pick a sequence $w_n = A \mu_n$ in $\imcone$ converging to some $w \in \R_+^\ndet$.
   Appealing to \eqref{positive}, we choose $c>0$ such that $\sum_{i=1}^\ndet a_i \geq c$ and write 
   \[ \mu_n(K) \leq \frac{1}{c} \int_K \Big(\sum_{i=1}^\ndet a_i\Big) \dd \mu_n = \frac{1}{c} \sum_{i=1}^\ndet (w_n)_i, \]
   where the last quantity is bounded since $(w_n)$ converges. Hence, the sequence $(\mu_n)$ is bounded, it has a weak-$\ast$ converging subsequence, say to some $\mu \in \MM_+$. Along the subsequence, $(A \mu_n)$ converges to $A \mu$ by weak-$\ast$ continuity of $A$, which proves that $w = A \mu \in \imcone$.
 \end{proof}
 
 The original optimisation problem~\eqref{opt} is related to the following optimisation problem:
 \begin{equation}
\label{cone_formulation}
\operatorname*{min}_{w \in \imcone}  \; D (y,w).
\end{equation}
More precisely, for any $w^\star$ optimal for the above problem, any measure $\mu^\star \in \MM_+$  such that $A \mu^\star = w^\star$ is optimal for the original problem.

\begin{lemma}
\label{unique_min}
The minimum in~\eqref{opt} is attained.
\end{lemma}
\begin{proof}
  We show that the optimal value in the minimisation problem~\eqref{cone_formulation} is attained.
  The function $w \mapsto D(y,w)$ is coercive and lower semicontinuous. Since the (non-empty) cone $\imcone$ is closed (\autoref{prop_cone_closed}), there exists an optimal $w^\star \in \imcone$ for the problem~\eqref{cone_formulation}. Any~$\mu^\star \in \MM_+$ such that $A\mu^\star =w^\star$ then provides a minimiser for the original problem~\eqref{opt}. 
\end{proof}

We now compute the (Lagrange) dual problem to the problem~\eqref{cone_formulation}. 
These computations turn out to be crucial in analysing results from numerical simulations, determining whether we should expect singularity or not for some given $y \in \R_+^\ndet$ from appropriately defined \emph{singularity certificates}.

We define the cone $\imcone^\ast := \setc*{\lambda \in \R^\ndet}{\scl{\lambda}{w} \geq 0, \;\, \forall w \in \imcone}$ dual to $\imcone$, which can be characterised as in~\cite{Georgiou2005} by \[\imcone^\ast = \setc*{\lambda \in \mathbb{R}^\ndet}{A^* \lambda \geq 0 \text{ on } K}.\]
The dual problem writes 
\begin{equation}
\label{dual_problem}
\operatorname*{max}_{\lambda \in \imcone^\ast}  \; g(\lambda)
,  
\end{equation}
where the function \(g \colon \R^\ndet \to \R\) is defined for $\lambda \in \imcone^\ast$ by
\begin{equation}
  \label{eq:defg}
  g(\lambda) := \min_{ w \in \R_+^\ndet} D(y,w)- \langle \lambda , w \rangle
  .
\end{equation}
Note that strong duality holds since the problem is convex and Slater's condition is obviously satisfied since the minimisation occurs over a cone in finite dimension.

The point of taking the Lagrange dual is that, since $D$ decomposes, so does $g$ and we find
\[ g(\lambda) = \sum_{i=1}^m \min_{w_i \geq 0} \paren[\big]{d(y_i,w_i)- \lambda_i w_i}. \]
As a result, defining
\begin{equation}
  \label{eq:defh}
  h(y, \lambda) := 
  \min_{w \geq 0} \left(d(y,w)- \lambda w \right)
\end{equation}
for a scalar $y \geq 0$, the resulting dual function will take the form
\[g(\lambda) = \sum_{i=1}^m h(y_i, \lambda_i).\]
The explicit computation of the function $h$ in \eqref{eq:defh} is carried out for the case of $\beta$-divergences in \autoref{appA}.

 \subsection{Optimality conditions}
We now use Fenchel duality to compute the optimality conditions for problem~\eqref{opt}. The assumption that the operator $A$ is continuous in the weak-$\ast$ topology plays a crucial role. 



\begin{proposition}
Let $\mu^\star$ be an optimal measure for~\eqref{opt}. 
Then, there exists
$\lambda^\star \in \partial D(y,\cdot)(A \mu^\star)$ such that
\beq 
\label{kkt}
A^\ast \lambda^\star \geq 0 \text { on } K, \qquad A^\ast \lambda^\star = 0 \text{ on } \supp(\mu^\star).
\eeq
\end{proposition}
\begin{proof}
We use the Fenchel--Rockafellar Theorem~\cite{Rockafellar1966}, with the optimisation problem rewritten as 
\[\min_{\mu \in \MM} D(y, A\mu) + \delta_{\MM_+}(\mu),\]
where $\delta_{\MM_+}$ is the indicator function in the convex analytic sense, i.e. $\delta_{\MM_+}(\mu) = 0$ if $\mu \in \MM_+$ and $+\infty$ otherwise.

In the context of the Fenchel--Rockafellar Theorem we need \emph{paired spaces}.
The natural choice is
$\mathcal{C}$ (endowed with its strong topology) and $\MM$ (endowed with its weak-$\ast$ topology).
We may apply the theorem thanks to the hypothesis~\eqref{blanket} and we obtain
\[\text{$\mu^\star$ is optimal} \quad \Longleftrightarrow \quad 0 \in A^\ast \partial D(y,\cdot) (A \mu^\star)  + N_{\MM_+} (\mu^\star),\]
where $N_{\MM_+} (\mu)$  is the normal cone of $\MM_+$ at $\mu$, defined by \[N_{\MM_+} (\mu):= \setc*{f \in \mathcal{C}}{\forall \nu \in \MM_+, \; \scl{\mu-\nu}{f} \geq 0}.\]
The normal cone can be identified as \[ N_{\MM_+}(\mu) = \setc*{ f \in \mathcal{C}} {f\leq 0 \text{ on } K, \; f = 0 \text{ on } \supp(\mu)},\]
see~\cite{Verdier2020}.
Picking $\lambda^\star \in \partial D(y,\cdot)(A \mu^\star)$ such that $A^\ast \lambda^\star \in -N_{\MM_+} (\mu^\star)$, we  exactly obtain~\eqref{kkt}.
\end{proof}
Note that for $\mu^\star$ optimal and with the notations of the above result, the separable form of $D$ ensures
\begin{equation} 
\label{subdiff_splitting} 
\lambda^\star \in \partial D(y,\cdot)(A \mu^\star) \quad \iff \quad\forall i \in \{1, \ldots, \ndet\},\; \lambda_i^\star \in \partial  d(y_i, \cdot)((A \mu^\star)_i).\end{equation}

\begin{remark}
\label{Separation}
Let us then emphasise the following important (yet straightforward) property: if $\mu^\star$ is optimal, then any $\lambda^\star \in \partial D(y,\cdot)(A \mu^\star)$ is such that
\[y_i \neq (A \mu^\star)_i \quad \implies \quad \lambda_i^\star \neq 0.\]
Indeed, assume $\lambda_i^\star = 0$. Then the inclusion $\lambda_i^\star = 0 \in \partial d(y_i, \cdot)((A \mu^\star)_i)$ is equivalent to $(A \mu^\star)_i$ minimising $d(y_i, \cdot)$, which by the separation property enforces $y_i = (A\mu^\star)_i$. 
\end{remark}

 \subsection{Singularity theorems}
 We now address the following question 
 \begin{quote}
 {If the data \(y\) is not in the image cone \(\imcone\), when do the optimality conditions~\eqref{kkt} lead to singular, or completely singular measures?}
 \end{quote} 
  By singular, we recall that we mean measures that have a singular part with respect to the Lebesgue measure, where completely singular measures are singular measures which furthermore have a zero absolutely continuous part, as per Definition~\ref{def:sparse}.
In order to prove that a given measure $\mu$ is completely singular, our approach will be to show that $\mu \neq 0$ and that $\supp(\mu)$ has zero Lebesgue measure.

 We begin with a result ensuring that all optimal measures are completely singular (and not only singular), but which holds only under the restrictive assumption that the detector functions $a_i$ be analytic. This generalises the result of~\cite{Verdier2020} to general divergences.
 
\begin{proposition}
\label{analytic_case}
Assume that the functions $(a_i)_{i = 1,\ldots, \ndet}$ are linearly independent in~$\mathcal{C}(K)$, analytic in $\mathrm{int}(K)$.

Then, if $y \notin \imcone$, any optimal measure, if it is non zero, is completely singular. 
\end{proposition}

Note that condition~\eqref{H} is one sufficient condition that $D$ can satisfy to ensure that all optimal measures $\mu^\star$ satisfy $\mu^\star \neq 0$, as the proof of~\autoref{sparse} shows.
 
 \begin{proof}
Let us pick some optimal measure $\mu^\star \neq 0$. 
The optimality conditions~\eqref{kkt} provide $\lambda^\star$ such that
$\lambda^\star \in \partial D(y,\cdot)(A \mu^\star)$ and $A^\ast \lambda^\star = 0$ on $\supp(\mu^\star)$. In other words, $\supp(\mu^\star)$ is contained in the zeros of the continuous function $\varphi^\star := A^\ast \lambda^\star = \sum_{i=1}^\ndet \lambda_i^\star a_i$. 

Since $y \notin \imcone$, there exists $i$ such that $y_i \neq (A \mu^\star)_i$. Using Remark~\ref{Separation}, this implies $\lambda_i^\star \neq 0$. Now, the linear independence of the functions $a_i$ shows that $\varphi^\star$ is not identically zero. 
Since the function is by assumption analytic on $\mathrm{int}(K)$, its zero set intersected with $\mathrm{int}(K)$ is of zero Lebesgue measure. Note that $\partial K$ has zero measure since $\mathrm{int}(K)$ is assumed to be a Lipschitz bounded connected open set. Hence the support of~$\mu^\star$ has zero Lebesgue measure, which concludes the proof.
 \end{proof}

 The main reason why this theorem is not satisfactory is that in many applications, the detector functions $a_i$ are compactly supported inside $K$, even though they may be analytic (or piecewise analytic) on their support.
 This is why we relax the analyticity assumption and consider now detector functions $a_i$ which are only \emph{piecewise} analytic.
 Dealing with this more complicated case also requires restricting our attention to divergences which satisfy the property~\eqref{H}.

We make some assumptions on the detector functions $a_i$, which require some notations.
First,  we define
\[J_i := \setc[\big]{j \in \set{1, \ldots, \ndet}}{\mathrm{supp}(a_j) \cap  \mathrm{supp}(a_i)\neq \emptyset}
  \qquad
  i \in \set{1, \ldots, \ndet}
  .
\]
It is the set storing which detector functions $a_j$ are active on $\supp(a_i)$.
In particular, we have $i \in J_i$.

We consider the following condition, which combines \emph{local linear independence} and \emph{piecewise analyticity}:
\begin{equation}
\label{C}
\text{there exists a partition $K = \bigcup_{k = 1}^r \overline{\Omega}_k$ such that}
\tag{C}
\end{equation}
\begin{itemize}
\item the sets $\Omega_k$, $k \in \{1,\ldots, r\}$ are Lipschitz open connected sets,
\item the detector functions $a_i$, $i  \in  \{1, \ldots, m\}$ are piecewise analytic on the partition, i.e., for each $i \in  \{1, \ldots, m\}$, $k \in \{ 1, \ldots, r\}$, $a_i$ is analytic on $\Omega_k$,
\item for all $i \in \{1, \ldots, m\}$, the family $(a_j)_{j \in J_i}$ is linearly independent in $\mathcal{C}(\Omega_k)$ for any $k$ such that $\Omega_k \cap \mathrm{supp}(a_i)\neq \emptyset$.
\end{itemize}

Then, the following holds.
\begin{theorem}
\label{sparse}
Assume that the divergence $D$ satisfies~\eqref{H}, and that the functions $(a_i)_{i = 1,\ldots, \ndet}$ satisfy~\eqref{C}.
Then, if $y \notin \imcone$, any $\mu^\star$ optimal is singular. 

More precisely, for any $i_0 \in \{1, \ldots, m\}$ such that $y_{i_0} \neq (A\mu^\star)_{i_0}$, $\mu^\star_{|\mathrm{supp}(a_{i_0})}$ is completely singular.
\end{theorem}
\begin{proof}
Let $\mu^\star$ be optimal. We may write the optimality conditions~\eqref{kkt}, namely
\[A^\ast \lambda^\star \geq 0 \text{ on } K, \quad A^\ast \lambda^\star= 0 \text{ on } \supp{(\mu^\star)},\]
where $\lambda_i^\star \in \partial \, d(y_i, \cdot)((A \mu^\star)_i)$ for all $i$ from~\eqref{subdiff_splitting}.

We shall prove that the following key property holds:
\begin{equation}
\label{propertyP}
\tag{$\mathcal{P}$}\forall i \in \{1, \ldots, m\}, \quad  y_i \neq (A \mu^\star)_i \implies (A \mu^\star)_i > 0.
\end{equation}
Indeed, let us pick $i_0$ such that $(A \mu^\star)_{i_0} = 0$; we need to show that $y_{i_0} =(A \mu^\star)_{i_0}  = 0$.
If we had $y_{i_0}>0$, we would find 
\[\lambda_{i_0} \in \partial d(y_{i_0}, \cdot)((A \mu^\star)_{i_0}) = \partial d(y_{i_0}, \cdot)(0),\]
 contradicting the emptiness of the subdifferential as given by~\eqref{H}, and  hence proving~\eqref{propertyP}.


Continuing with $i_0 \in\{1, \ldots, m\}$ such that $y_{i_0} \neq (A \mu^\star)_{i_0}$ and hence such that $\lambda_{i_0}^\star \neq 0$ as explained in Remark~\ref{Separation}, Property \eqref{propertyP} entails
\begin{equation}
\label{non-empty}  
\mathrm{supp}(\mu^\star) \cap  \mathrm{supp}(a_{i_0})  \neq \emptyset,
\end{equation}
which in particular shows that $\mu^\star_{|\mathrm{supp}(a_{i_0})} \neq 0$.
The optimality of $\mu^\star$ also ensures 
\[ \mathrm{supp}(\mu^\star)  \subset \setc[\bigg]{x \in K}{\sum_{j=1}^m \lambda_j^\star a_j(x) = 0}
  .
\]
Hence, 
we have
\[\mathrm{supp}(\mu^\star)  \cap  \mathrm{supp}(a_{i_0})  \subset \setc[\bigg]{x \in \mathrm{supp}(a_{i_0})}{\sum_{j \in J_{i_0}} \lambda_j^\star a_j(x) = 0}
  ,
\]
and the set on the right-hand side cannot be empty owing to~\eqref{non-empty}. We denote it~$K_{i_0}$.

Now let us have a closer look at this set.
Suppose that for some \(k\), the set $\overline{\Omega}_k$ intersects $\supp(a_{i_0})$.
The set $K_{i_0}$ may contain parts of the edges $\partial \Omega_k$.
Other than that, its intersection with $\Omega_k$ cannot have positive Lebesgue measure.
If it were the case, analyticity would imply $\sum_{j \in J_{i_0}} \lambda_j^\star a_j = 0$ identically on $\Omega_k$.
Then, linear independence would enforce $\lambda_j^\star=0$ for all $j \in J_{i_0}$, contradicting $\lambda_{i_0}^\star \neq 0$.
Hence, the support of $\mu^\star_{|\mathrm{supp}(a_{i_0})}$ is included in a finite union of zero-measure sets (lying in the interior of the partition pieces) and of $(p-1)$-dimensional edges (lying at their boundary). As a result, the support of $\mu^\star_{|\mathrm{supp}(a_{i_0})}$ is of measure zero, which concludes the proof.
\end{proof}

\begin{remark}
We reached the conclusion that if $y \notin A(\MM_+)$, any optimal measure~$\mu^\star$ will be such that $\mu^\star_{|\mathrm{supp}(a_{i_0})}$ is not the zero measure and has support of Lebesgue measure zero for at least some $i_0 \in \{1, \ldots, m\}$. In fact, one can be a bit more specific and assert that $\mu^\star_{|\mathrm{supp}(a_{i_0})}$  has support included in a union of subvarieties, each of them having dimension at most $p-1$~\cite[Theorem~6.3.3]{Krantz2002}.
\end{remark}

\begin{remark}
\label{Condition_Gradient}
We emphasise that~\autoref{sparse} does use the fact that $\lambda^\star \in \partial D(y,\cdot)(A \mu^\star)$, the optimality condition $A ^\ast \lambda^\star = 0$ over $\mathrm{supp}(\mu^\star)$, but it does not use the optimality condition $A^\ast \lambda^\star \geq 0$ over $K$. The second optimality condition combined with the first one indeed yields
\[\mathrm{supp}(\mu^\star)  \subset \argmin \paren[\bigg]{\sum_{j=1}^m \lambda_j^\star a_j}.\]
This in turn can only make measures ``more singular''.
For instance, one has for all~$k \in \{ 1, \ldots, r\}$
\[ \mathrm{supp}(\mu^\star)  \cap \Omega_k \subset \setc[\bigg]{x \in \Omega_k}{\sum_{j=1}^m \lambda_j^\star \nabla a_j(x) = 0}
  .
\]

\end{remark}

\section{Detector Functions $a_i$ leading to singularity}
\label{sec4}
We investigate a few examples of detector functions and discuss whether they  satisfy Property (C) or not. We start with a toy example and then consider a situation which more closely matches functions encountered in practice (in the case of PET).

\subsection{Shifted polynomials in dimension $1$}

Let $\R_k[X]$  denote the vector space of real polynomials of degree at most~$k \geq 0$ over $\R$.
\begin{lemma}
\label{ind}
Let $P \in \R_k[X]$ with $\mathrm{deg}(P) = k$, and $z_1, \ldots, z_r$ be distinct in $\R$. Then
\[\text{ the family 
$(P(\cdot + z_1), \ldots, P(\cdot + z_r))$ is linearly independent in $\R_k[X]$ if $r \leq k+1$}.\]
\end{lemma}

\begin{proof}
  Since $P$ has degree \(k\), the set of vectors \[e_{\alpha} := {\frac{D^\alpha P}{\alpha !}} \qquad {0 \leq \alpha \leq k}\] is a basis of $\R_k[X]$.
  Now, using the Taylor formula for polynomials, we obtain
  \[ P(X+z) = \sum_{\alpha = 0}^k z^\alpha \frac{D^\alpha P}{\alpha !}(X) = \sum_{\alpha=0}^k z^{\alpha} e_{\alpha}(X)
    \qquad
    z \in \RR
    .
  \]
If we express $(P(\cdot + z_1), \ldots, P(\cdot + z_r))$ in the basis \(e_{\alpha}\), we thus obtain the Vandermonde matrix
$$
\left(\begin{array}{rrrrrr}
1 & 1  &\dots & 1\\
z_1 & z_2  &\dots & z_r
\vspace{0.1cm}
 \\
z_1^{2}&  z_2^{2} &\dots &  z_r^{2}\\ 
  \vdots &\vdots &\ldots& \vdots\\
z_1^{k}&  z_2^{k} &\dots &  z_r^{k}\end{array}
\right).
$$
Hence, since the shifts $z_1, \ldots, z_r$ are distinct, this matrix is of rank $r$ as soon as $r \leq  k+1$, which concludes the proof.
\end{proof}

Now, consider some univariate polynomial $P \in \R_k[X]$ such that $P > 0$ on some interval $(a,b)$, with $P(a) = P(b) = 0$. Let $a_0 = P \mathds{1}_{(a,b)}$ and define a countable family of shifted functions by
\[a_0(\cdot - l h),\quad  l \in \Z,\]
where $h>0$ is the grid spacing.
Suppose $K$ is some compact interval in $\R$, and~$\ndet$ stands for the number the above functions whose support intersects the interior of~$K$, which we denote $a_1, \ldots a_{\ndet}$. For an example, see~\autoref{fig:shifts}.
\begin{corollary}
\label{dim1}
Let $r := \ceil*{\frac{b-a}{h}}$ and assume $r \leq k+1$. Then the functions $(a_i)_{1 \leq i \leq \ndet}$ satisfy Assumption~\eqref{C}.
\end{corollary}
\begin{proof}
The functions $(a_i)_{1 \leq i \leq \ndet}$ are clearly piecewise analytic. Let us tackle the issue of local linear independence in more detail. We let $\Omega \subset K$ be any open interval such that all functions $a_i$ are analytic on $\Omega$. It is easily checked that there are at most~$r$ functions active on each such set $\Omega$. Hence, if $i$ is such that $\mathrm{\supp}(a_i)$ intersects $\Omega$, a linear combination of $(a_j)_{j \in J_i}$ writes 
\[ \sum_{\lambda_j \in J_i}\lambda_j P(\cdot + z_j),\] for some distinct shifts $z_j$ (which are multiples of $h$). This polynomial vanishes on~$\Omega$, and hence is the zero polynomial. We may then use Lemma~\ref{ind} to infer that $\lambda_j = 0$ for all $j \in J_i$, hence the local linear independence.
\end{proof}


\begin{figure}
  \centering
  \includegraphics[width=.8\textwidth]{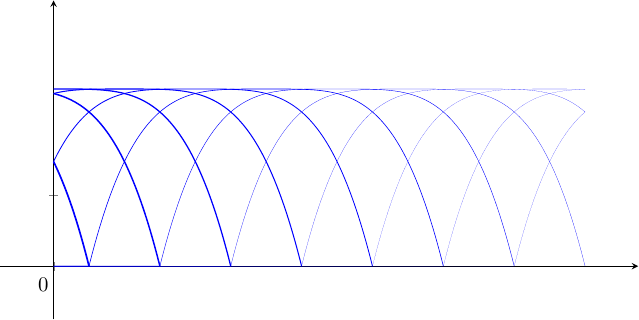}
  \caption{Example of shifts with $P(x) := 1 -x^4$, $a=-1, b=1$, $h = \frac{2}{5}$, shown over the compact $K = [0,3]$. The resulting functions satisfy Assumption~\eqref{C}, as per Corollary~\ref{dim1}.
  }
  \label{fig:shifts}
\end{figure}

\subsection{PET functions on a regular polygon}
\label{sec:petfunc}
By \textit{PET functions}, we informally refer to the functions associated to a continuous-discrete PET forward operator. In view of illustrating why PET functions are expected to satisfy Assumption~\eqref{C}, we focus on a representative example of PET, the 2D case with detectors regularly placed on the unit circle. 

Let us review the  basics of PET physics: $a_i(x)$ stands for the probability that an emission  that occurred at position $x$ leads to a detection by the $i$th pair of detectors.
More precisely, a positron emitted at position $x$ interacts with an electron, almost instantly and hence at a position which can be approximated to be $x$. This leads to the emission of two photons in opposite directions, where the direction is uniformly chosen at random. These two photons are then (almost) simultaneously detected by a pair of detectors.

\subsubsection{Modelling for PET functions}
Let $\theta = \frac{2 \pi}{N}$ for some $N \in \mathbb{N}^*$.
Working with complex notations, we denote $z_k = \ee^{\ii k \theta}$, $k = 0, \ldots, N-1$. Define $P_{N}$ to be the regular polygon associated to the points $z_k$, which lies inside the unit disk. Throughout this section, $T(u,v,w)$ stands for the (closed) triangle defined by three non-aligned complex numbers $u$, $v$, $w$, and $\mathrm{int}(B)$ denotes the interior of a set~$B \subset \mathbb{C}$.

For convenience, we use two indices $j,k$ to index functions, letting $a_{j,k}$ be the function associated to the pair of detectors formed by the line segments $[z_j, z_{j+1}]$ and $[z_k, z_{k+1}]$ for  $0 \leq j < k \leq N-1$. 
There are $\textstyle \ndet = \frac{1}{2} N(N-1)$ such functions.

Given two such segments and $z \in P_N$,  $a_{j,k}(z)$ equals the probability that a line passing through $z$ intersects both line segments, when the line orientation is chosen uniformly at random. This definition is ambiguous at the vertices $z_k$ since these points belong to two detectors. In fact, one can decide for a convention at these points, but whatever the choice, some of the functions will not be continuous there, see also~\autoref{fig:petfuncview}. 

In practice, however, the object to be imaged will not reach the detectors: hence the compact $K$ to be considered will typically satisfy $K \subset  \mathrm{int}(P_N)$.  This in turn will ensure that we are in the framework of functions  in $\mathcal{C}(K)$.
In what follows, we shall thus define the functions $a_{j,k}$ over $\mathrm{int}(P_N)$. We note, however, that all the definitions given below are valid on $P_N \setminus\set{z_0, \ldots,z_{N-1}}$.

\subsubsection{Formulae for PET functions.}
We let $0 \leq j < k \leq N-1$ be fixed. Given the probabilistic model chosen for the PET functions, computing $a_{j,k}(z)$ boils down to computing angles (after proper normalisation). 

\textit{Adjacent detectors.} First assume that the two segments have one point in common. In other words, we are interested in the functions $a_{j, j+1}$, $0 \leq j \leq N-2$, and $a_{0,N-1}$. 

The functions $a_{j, j+1}$, $0 \leq j \leq N-2$ vanish outside of $T(z_j, z_{j+1}, z_{j+2})$. 
Inside the triangle and for $z \neq z_j, z_{j+2}$, the angle we are looking for is the (non-oriented) angle $(\overrightarrow{z_j z}, \overrightarrow{z z_{j+2}})$.
Hence, after normalisation and using the notation $\mathrm{Arg}$ for the principal argument, we find in complex notations
\begin{equation}
\label{adjacent_1}
\forall z \in T(z_j, z_{j+1}, z_{j+2}) \cap \mathrm{int}(P_N), \quad  a_{j,j+1}(z) = \frac{1}{\pi} \mathrm{Arg}\bigg(\frac{z_{j+2}-z}{z-{z_{j}}}\bigg).
\end{equation}
These functions are continuous on $\mathrm{int}(P_N)$ and analytic on $\mathrm{int}(T(z_j, z_{j+1}, z_{j+2}))$.

The above reasoning also covers the function $a_{0,N-1}$, which vanishes outside of $T(z_{N-1}, z_0, z_1)$, and 
\begin{equation}
\label{adjacent_1}
\forall z \in T(z_{N-1}, z_0, z_1)  \cap \mathrm{int}(P_N), \quad  a_{0,N-1}(z) = \frac{1}{\pi}  \mathrm{Arg}\bigg(\frac{z_{N-1}-z}{z-{z_{1}}}\bigg).
\end{equation}

\textit{Non-adjacent detectors.} Now, consider the remaining functions $a_{j,k}$, i.e., when $2 \leq k-j \leq N-2$. Then, $[z_j, z_{j+1}]$ and $[z_k, z_{k+1}]$ have no point in common and the lines $(z_{j+1}, z_k)$ and $(z_j, z_{k+1})$ are easily shown to be parallel. 
Again, $a_{j, k}$ is zero outside of the closed trapezium defined by $z_j, z_{j+1}, z_k, z_{k+1}$. 

With $c_{j,k}$ denoting the point at which the diagonals of the trapezium intersect, we find by reasoning as in the adjacent case
\begin{itemize}
\item If $z \in T(z_j, z_{j+1}, c_{j,k})  \cap \mathrm{int}(P_N)$, 
\begin{equation}
\label{non_adjacent_1}
a_{j,k}(z) = \frac{1}{\pi} (\overrightarrow{z_k z}, \overrightarrow{z_{k+1} z}) = \frac{1}{\pi} \mathrm{Arg}\bigg(\frac{z-z_{k+1}}{z-{z_{k}}}\bigg).
\end{equation}
\item If $z \in T(z_k, z_{k+1}, c_{j,k})  \cap \mathrm{int}(P_N)$, 
\begin{equation}
\label{non_adjacent_2}
a_{j,k}(z) =  \frac{1}{\pi} (\overrightarrow{z_j z}, \overrightarrow{z_{j+1} z}) = \frac{1}{\pi} \mathrm{Arg}\bigg(\frac{z-z_{j+1}}{z-{z_{j}}}\bigg).
\end{equation}
\item If $z \in T(z_{j+1}, z_k, c_{j,k}) \cap \mathrm{int}(P_N)$, 
\begin{equation}
\label{non_adjacent_3}
a_{j,k}(z) =  \frac{1}{\pi} (\overrightarrow{z_k z}, \overrightarrow{z z_{j+1}}) = \frac{1}{\pi} \mathrm{Arg}\bigg(\frac{z_{j+1}-z}{z-{z_{k}}}\bigg).
\end{equation}
\item If $z \in T(z_{k+1}, z_j, c_{j,k})  \cap \mathrm{int}(P_N)$, 
\begin{equation}
\label{non_adjacent_4}
a_{j,k}(z) =  \frac{1}{\pi} (\overrightarrow{z_j z}, \overrightarrow{z z_{k+1}}) =  \frac{1}{\pi} \mathrm{Arg}\bigg(\frac{z_{k+1}-z}{z-{z_{j}}}\bigg).
\end{equation}
\end{itemize}
One can check that these definitions are coherent where the triangles intersect. We note that these functions are continuous on $\mathrm{int}(P_N)$. Furthermore, they are analytic on all four open subtriangles $\mathrm{int}(T(z_j, z_{j+1}, c_{j,k}))$, $\mathrm{int}(T(z_k, z_{k+1}, c_{j,k}))$, $\mathrm{int}(T(z_{j+1}, z_k, c_{j,k}))$ and $\mathrm{int}(T(z_{k+1}, z_j, c_{j,k}))$.

\begin{remark}
In fact, the formulae~\eqref{non_adjacent_1}-\eqref{non_adjacent_2}-\eqref{non_adjacent_3}-\eqref{non_adjacent_4} derived above apply to any pair of detectors given by two line segments that do not intersect.  
\end{remark}
The situation is illustrated by~\autoref{fig:petfuncview} which shows the plot of a typical PET function between two given segments.~\autoref{Polygon} shows the supports of two PET functions in the case where $N=12$.
\begin{figure}[t]
  \centering
  \begin{subfigure}{.58\textwidth}
\includegraphics[width=\linewidth]{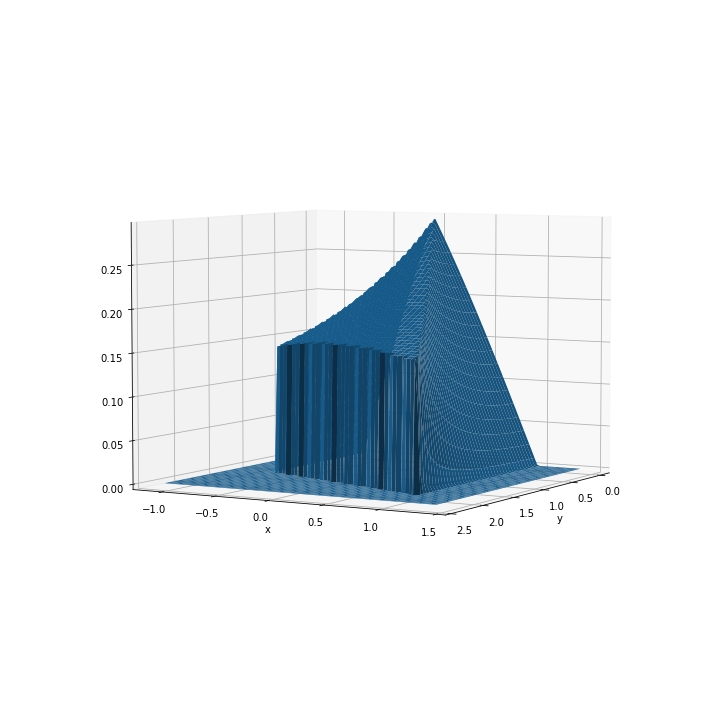}
  \end{subfigure}
  \begin{subfigure}{.4\linewidth}
    \includegraphics[width=\textwidth]{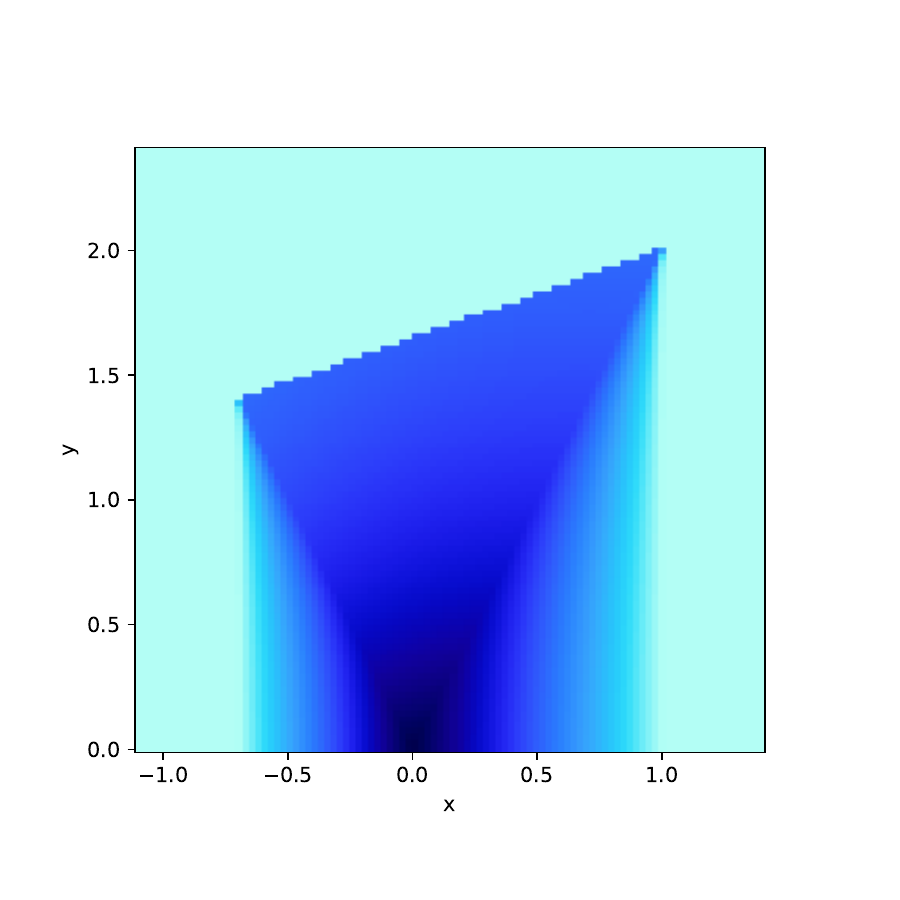}
  \end{subfigure}
  \caption{Two views of the functions described in \autoref{sec:petfunc}: here the two detectors are the line segment $[u,v]$ with $u=1+2i$, $v \sim -0.7 + 1.4i$, and the symmetric line segment $[\bar u, \bar v]$ with respect to the real axis.  
    We only plot the function on half of its domain since it is symmetric with respect to the real axis. As can be seen, the functions cannot be continuously extended up to the vertices.
  }
  \label{fig:petfuncview}
\end{figure}

\subsubsection{PET functions and Assumption~\eqref{C}.}
The angle interpretation shows that we obviously have $a_{j,k} \geq 0$ for all $j,k$, as well as the relation
\[\forall z \in \mathrm{int}(P_N), \qquad \sum_{0 \leq j < k \leq N-1} a_{j,k}(z) = 1,\]
hence PET functions satisfy the positivity assumption~\eqref{positive}  over any compact $K \subset \mathrm{int}(P_N)$.
By symmetry, we also have for any $j<k$, $j'< k'$ such that $k'-j' = k - j = r$
\[ \forall z \in \mathrm{int}(P_N), \qquad a_{j',k'}(z) = a_{j,k}(e^{-i r \theta} z),\]
meaning that all functions can be obtained from $a_{0,j}, \, j = 1, \ldots, N-1$, up to appropriate rotations, which is reminiscent of the previous toy example with functions all equal up to translation. Here, there is not $1$ but $N$ "base" functions from which all the others are deduced by rotation. Finally, we note that further (mirror) symmetries can be exploited, with respect to line segments linking a given vertex and its diametric opposite vertex when $N$ is even (or a given vertex and the middle point of the diametrically opposed detector when $N$ is odd).

From the above discussion, PET functions are in $\mathcal{C}(K)$ for any compact $K \subset \mathrm{int}(P_N)$. Furthermore, we may build a partition ensuring that they are piecewise analytic in the sense defined by Assumption~\eqref{C} in any compact $K \subset \mathrm{int}(P_N)$. 

A partition that ensures piecewise analyticity is easily built by drawing all line segments $[z_j, z_k]$ (and then intersecting with~$K$). This construction naturally isolates the four subtriangles associated to non-adjacent detectors. 

The issue of local linear independence (associated to such partitions) is more subtle. We conjecture that the PET functions are locally linearly independent, for any value of $N$ and any (reasonably large) compact $K \subset \mathrm{int}(P_N)$. To support this conjecture, let us consider the case of $N=4$ with $K\subset \mathrm{int}(P_4)$ any (reasonable) compact.
\begin{proposition}
For $N=4$, and any compact $K \subset \mathrm{int}(P_4)$ containing at least a neighbourhood of $0$, the PET functions are locally linearly independent over $K$.
\end{proposition}
\begin{proof}
There are $6$ PET functions for $N=4$, and the partition described above is obtained by considering the four quadrants intersected with $K$. By symmetry, we will be done if we prove the result for only one of the four: we focus on the first quadrant, i.e. \[\Omega_1 :=\mathrm{int}\left(K \cap \left\{z \in \mathbb{C}, \; \mathrm{Re}(z) > 0, \mathrm{Im}(z) > 0\right\}\right).\]
Over $\Omega_1$, there are $4$ active PET functions, namely $a_{0,1}$, $a_{0,2}$, $a_{0,3}$ and $a_{1,3}$. Consider a vanishing linear combination, i.e. $\lambda_{0,1}$, $\lambda_{0,2}$, $\lambda_{0,3}$ and $\lambda_{1,3}$ real numbers such that
\[\forall z \in \Omega_1, \quad \lambda_{0,1} a_{0,1}(z) + \lambda_{0,2} a_{0,2}(z) + \lambda_{0,3} a_{0,3}(z) + \lambda_{1,3} a_{1,3}(z)= 0.\]
In fact, this equality extends to the boundary of $\Omega_1$ by continuity of the functions on $K$. 

We successively evaluate the equality at $z = 0$, any $z \neq 0$ on the real axis, any $z \neq 0$ on the imaginary  axis  and any $z$ on the first diagonal.

For $z =0$, $a_{0,1}(z) = a_{0,3}(z) = 0$ and by symmetry $a_{0,2}(z) = a_{1,3}(z) \neq 0$, hence we find $\lambda_{0,2} + \lambda_{1,3} = 0$. For $z \neq 0, z \in \Omega_1$ on the real axis, $a_{0,1}(z) = 0$, and by symmetry  $a_{0,2}(z) = a_{1,3}(z) \neq 0$, hence $\lambda_{0,3} = 0$ since $a_{0,3}(z) \neq 0$. Similarly with $z \neq 0, z \in \Omega_1$ on the imaginary axis, we uncover $\lambda_{0,1} = 0$. Finally, one can check that the function $t \mapsto a_{0,2}(z(t))-a_{1,3}(z(t))$ with $z(t) = (1+i)t$, i.e., on the first diagonal, has a positive derivative at $t=0$. Hence $a_{0,2}(z)> a_{1,3}(z)$ for any  $z \in \Omega_1, z \neq 0$ on the first diagonal sufficiently close to $0$. Picking such an element $z$ leads to $\lambda_{0,2} a_{0,2}(z) + \lambda_{1,3} a_{1,3}(z) = 0$ which, together with $\lambda_{0,2} + \lambda_{1,3} = 0$, implies $\lambda_{0,2} = \lambda_{1,3}=0$  and concludes the proof.
\end{proof}

\begin{figure}[h!]
\centering{
\includegraphics[width=0.6\linewidth]{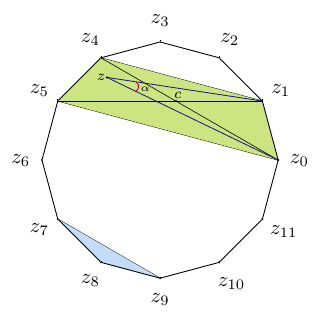} 
}
  \caption{Example of PET functions for the regular polygon~$P_N$, with $N= 12$. Outside of the pale blue region, $a_{7,8}$ vanishes. Outside of the pale green region (a trapezium whose diagonals intersect at the point denoted $c = c_{0,4}$), the function $a_{0,4}$ vanishes.
At the point~$z$, $a_{0,4}(z)$ is given by $\frac{\alpha}{\pi}$.}
  \label{Polygon}
\end{figure}

 \section{Numerical experiments}
  \label{sec5}
We here consider some simulations of algorithms solving~\eqref{opt} in different contexts, where the results exhibit singularity  as expected from the theoretical results. 
 \subsection{Singularity Certificates}
 As evidenced by our results, under the  Assumptions~\eqref{H} and  Assumption~\eqref{C}, the relevant criterion for singularity of optimisers of~\eqref{opt} is independent of the specific divergence $D$, as the question reduces to: 
 \[\text{do we have } y \in \imcone?\]
In practice, as one wants to solve~\eqref{opt} (or possibly a regularised version thereof) in the form of some iterative algorithm defined by iterates of the form
\[ \mu_{k+1} = G_k(\mu_k),\]
we are looking for methods allowing us to \emph{guarantee} that $y \notin \imcone$.
Our strategy is to devise a method to prove that $y \notin \imcone$ which writes as a function of $\mu_k$ and that gets better as $k \to +\infty$.

One approach towards this is to make use of duality: by weak duality (and recalling the definition of \(g\) in \eqref{eq:defg}) we always have
\[ \forall \mu \in \MM_+, \; \forall \lambda \in \imcone^\ast, \quad \loss(\mu) \geq g(\lambda)\]
where we recall that $\loss$ refers to the functional to be minimised $\loss(\mu) = D(y,A \mu)$ as per definition~\eqref{loss}.
Since $y \in \imcone$ if and only if $\min_{\mu \in \MM_+} \loss(\mu) = 0$, this entails the following straightforward result: 
\beq
\label{certificate}
(\exists \lambda \in \imcone^\ast, \; g(\lambda) > 0)  \quad \implies \quad y \notin \imcone.
\eeq
We will call a vector $\lambda \in \imcone^\ast$ such that $g(\lambda) > 0$ a \emph{dual certificate of singularity}.

\begin{remark}
We note that checking whether $\lambda \in \imcone^\ast  \iff A^\ast \lambda \geq 0$ for a given $\lambda \in \R^\ndet$ can easily be done at the discrete level, {i.e.}, for the discretised version of the operator $A$. However, it does not mean that $A^\ast \lambda \geq 0$ holds at the continuous level; one would need to control discretisation errors in order to ensure such an inequality. When performing numerical experiments, we did not account for them. Hence, whenever we will have claimed to have found a dual certificate, this will be abusive and will refer to discrete certificates.  
\end{remark}

This provides a natural method when it comes to establishing singularity: assume we have a convergent algorithm for solving~\eqref{opt}, in the sense that each subsequence of $(\mu_k)_{k \in \mathbb{N}}$ has subsequence that converges (in the weak-$\ast$ sense) to some minimiser $\mu^\star$ of~\eqref{opt}.

Oftentimes, one can derive an explicit link between primal and dual variables from the relation
\[\lambda^\star \in \partial D(y, \cdot)(w^\star),\]
which we know hold for optimal $w^\star$ under the hypotheses of~\autoref{sparse}. For instance, in the case of $\beta$-divergences,  an optimal dual variable $\lambda^\star \in \imcone^\ast$ is related to a primal variable $w^\star$ by $\lambda^\star := (w^\star)^{\beta - 2} (w^\star - y)$ where powers and multiplications (or divisions) are to be understood componentwise. Then, a candidate of choice for a dual certificate is given by
\[\lambda_k := (A \mu_k)^{\beta - 2} (A \mu_k - y).\] 
In particular, $\lambda_k$ will converge to $\lambda^\star$ along subsequences, hence the convergence of~$g(\lambda_k)$ to $\max_{\lambda \in \imcone^\ast} g(\lambda) = \min_{\mu \geq 0} \loss(\mu)$, the last equality being valid since strong duality obtains. Hence, if $y \notin \imcone$, we should have $\lim g(\lambda_k) > 0$ as $k \to + \infty$.


A caveat with our choice is that we should only have dual admissibility $\lambda_k \in \imcone^\ast$ at the limit $k \to +\infty$, and not for a fixed iteration number $k$. In practice, we take $k$ large and if $\lambda_k \notin \imcone^\ast$, we set $\tilde \lambda_k := \lambda_k  + c$ where $c >0$ is the smallest constant restoring  dual admissibility, estimated by bisection. In other words, we choose the minimal $c>0$ such that $A^\ast \tilde \lambda_k \geq 0$, {i.e.}, $\tilde \lambda_k \in \imcone^\ast$. Such a constant~$c$ exists because $A^\ast (\lambda_k +c) =  A^\ast \lambda_k +  c\, A^\ast 1$ and $A^\ast 1 = \sum_{i=1}^\ndet a_i>  0$ over~$K$ by assumption~\eqref{positive}.

\subsection{Emission Tomography Example}
We first look at an example from PET, where the aim is to solve~\eqref{opt}. We are investigating whether optimisers are singular depending on  the  noise level. Indeed, our conjecture that PET functions satisfy hypothesis (C) means that singular measures should be obtained, at least for a sufficient amount of noise.

When the divergence used is a $\beta$-divergence $D_\beta$  with $\beta \in [1,2]$, a common way to solve the optimisation problem~\eqref{opt} is to use the following iterates, called \emph{multiplicative}~\cite{Lee2001, Fevotte2011}. Starting from some $\mu_0 \in \MM_+$ (typically with a positive constant density over the domain), the iterates write
 \begin{equation}
 \label{multiplicative_updates}
 \mu_{k+1} = \mu_k \frac{A^\ast((A\mu_k)^{\beta - 2} y)}{A^\ast((A\mu_k)^{\beta - 1})},
 \end{equation}
assuming that one can prove $\textstyle \frac{A^\ast((A\mu_k)^{\beta - 2} y)}{A^\ast((A\mu_k)^{\beta - 1})} \in \mathcal{C}$ along iterates.

These iterates have the ML-EM algorithm ($\beta = 1$) and the Iterative Image Space Reconstruction ($\beta = 2$) as particular cases~\cite{Depierro1993}, and proofs of convergence for these algorithms with any $\beta \in [1,2]$ can be found in~\cite{Yang2011}, in the finite-dimensional case.
One advantage of these algorithms is the decrease of the functional $\loss$ along iterates, see the proof in~\cite{Fevotte2011} in the finite-dimensional setting.


We now present the results of applying the algorithm in the case of a 2D PET operator $A$ with \num{90} views and \num{64} tangential positions (hence, $\ndet = 90 \times 64 = 5760$).
We run the simulations using the Operator Discretization Library~\cite{Adler2017}, a Python wrapper around the \textsc{Astra} toolbox~\cite{astra}.

The image \(\mu\) is the Derenzo phantom, denoted $\mu_r$.
The data is obtained by (re-scaled) Poisson draws, with a time-variable (or dose-variable)~$t$ which accounts for the level of noise.
In other words, $y \sim \frac{1}{t} \mathcal P(t A \mu_r)$, and the higher~$t$, the lower the noise.
In order to approach the infinite-dimensional setting of our work, we increase the resolution to $512 \times 512$ pixels.

Finally, we take $\beta = 1.2$, on purpose not quite matching the noise statistics, as Poisson noise should lead one to take $\beta = 1$. We hence  mimic the situation of not knowing the exact noise statistics.

In \autoref{time0} and \autoref{time-1}, we display the evolution of the loss function $\loss$ along iterates, {i.e}, $k \mapsto \loss(\mu_k)$, starting from $\mu_0 = 1$, for noise levels \(t=1\) and \(t=10^{-1}\) respectively.
As expected, the function decreases. We also plot the maximum attained for each reconstruction, namely $k \mapsto \max(\mu_k)$, which tends to increase. Finally, we show the reconstruction after $k=100$ and $k=1000$ iterates. 

In the noisier case $t=1$, some pixels clearly take over as one keeps iterating.
Moreover, we can guarantee that we should indeed expect singularity, as we may provide a dual certificate proving that the data \(y\) is not in the image cone \(\imcone\).
This is also suggested by the fact that $k \mapsto \loss(\mu_k)$ seems to converge to a positive value rather than to zero.

\begin{figure}[h!]
\centering{
\begin{subfigure}{0.49\textwidth}
  \begin{tikzpicture}[scale=.7]
    \begin{axis}[
      xlabel={Iterations},
      ylabel={Loss function},
      grid=major,
      xmin=0,
      ymin=0,
      xmax=1000,
      ]
      \addplot[blue, line width=2pt] table {Figures/exp0/divs.data};
    \end{axis}
  \end{tikzpicture}
\caption{} \label{}
\end{subfigure} 
\hfill
\begin{subfigure}{0.49\textwidth}
  \begin{tikzpicture}[scale=.7]
  \begin{axis}[
    xlabel={Iterations},
    ylabel={Maximum},
    grid=major,
    xmin=0,
    ymin=0,
    xmax=1000,
    axis y line=right,
    ]
    \addplot[blue, line width=2pt] table {Figures/exp0/maxs.data};
  \end{axis}
\end{tikzpicture}
\caption{} \label{}
\end{subfigure} 
\\
\begin{subfigure}{0.49\textwidth}
\includegraphics[width=1\linewidth]{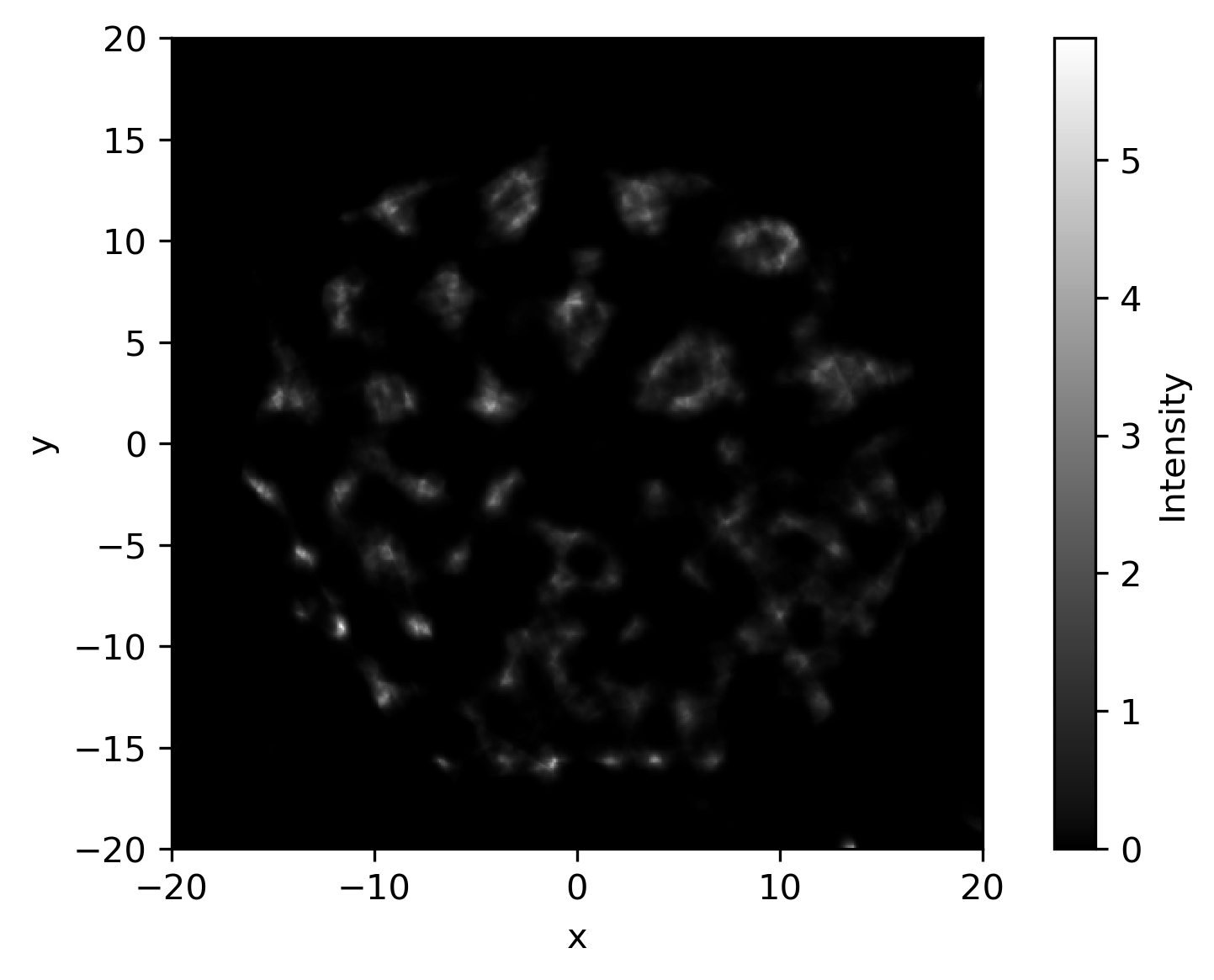} 
\caption{} \label{}
\end{subfigure} 
\hfill
\begin{subfigure}{0.49\textwidth}
\includegraphics[width=1\linewidth]{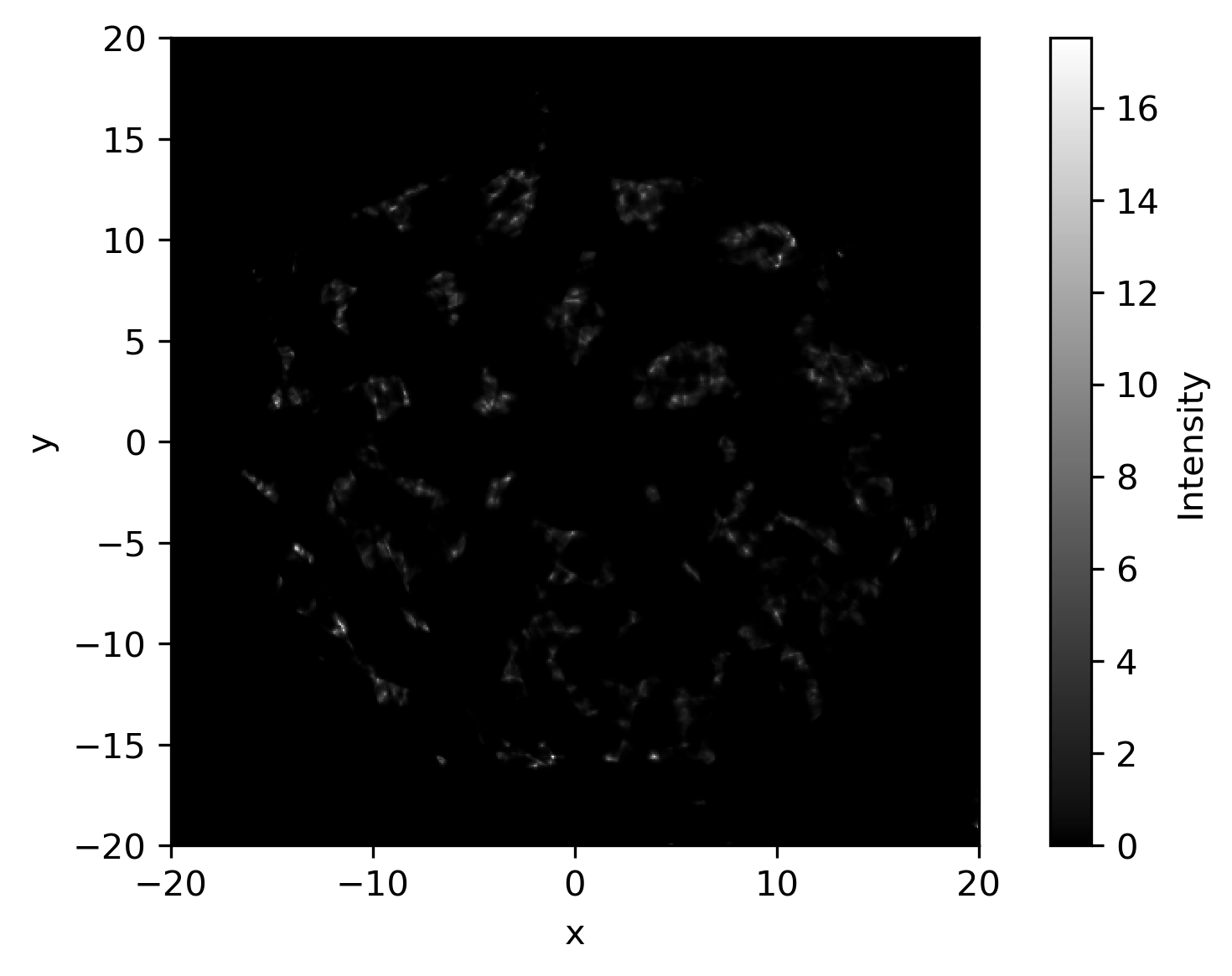} 
\caption{} \label{}
\end{subfigure} 
}
  \caption{Case $t=1$. (A) Divergence along iterates. (B) Maximum of reconstruction along iterates. (C) Reconstruction after $100$ iterates. (D) Reconstruction after $1000$ iterates.
  }
  \label{time0}
\end{figure}

In the less noisy case $t=10^{-1}$, it seems that the loss $\loss(\mu_k)$ is not converging to zero, which may be a hint that we should also expect singularity.
However, in this case we are not able to guarantee it with dual certificates.
Note that if singularity were to be true in this case as well, the reason could be that  many more iterates are needed to ascertain its presence.

In fact, we conjecture that singularity does arise at this noise level. But, as our theoretical results suggest, singularity is expected on specific parts of the domain. This is what seem to be observed for medium noise: although some few pixels take larger values along iterates, the rest of the image remains rather smooth. As a result, cropping the images to some value is a good practical solution to mitigating the night-sky effect.

\begin{figure}[h!]
\begin{subfigure}{0.49\textwidth}
  \begin{tikzpicture}[scale=.7]
    \begin{axis}[
      xlabel={Iterations},
      ylabel={Loss function},
      grid=major,
      xmin=0,
      ymin=0,
      xmax=1000,
      ]
      \addplot[blue, line width=2pt] table {Figures/exp1/divs.data};
    \end{axis}
  \end{tikzpicture}
\caption{} \label{}
\end{subfigure} 
\hfill
\begin{subfigure}{0.49\textwidth}
  \begin{tikzpicture}[scale=.7]
    \begin{axis}[
      xlabel={Iterations},
      ylabel={Maximum},
      grid=major,
      xmin=0,
      ymin=0,
      xmax=1000,
      axis y line=right,
      ]
      \addplot[blue, line width=2pt] table {Figures/exp1/maxs.data};
    \end{axis}
  \end{tikzpicture}
\caption{} \label{}
\end{subfigure} 
\\
\begin{subfigure}{0.49\textwidth}
  \includegraphics[width=1\linewidth]{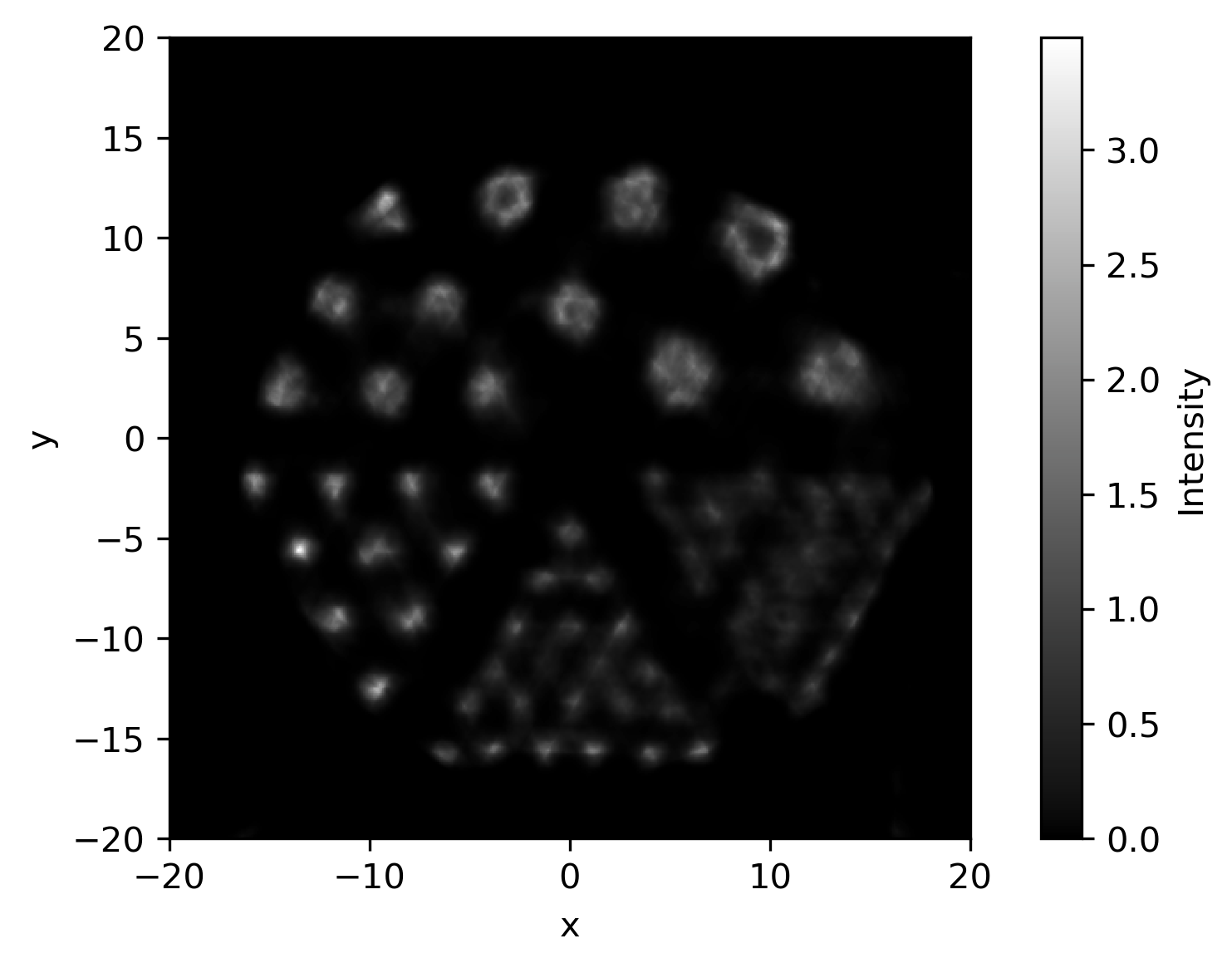} 
\caption{} \label{}
\end{subfigure} 
\hfill
\begin{subfigure}{0.49\textwidth}
  \includegraphics[width=1\linewidth]{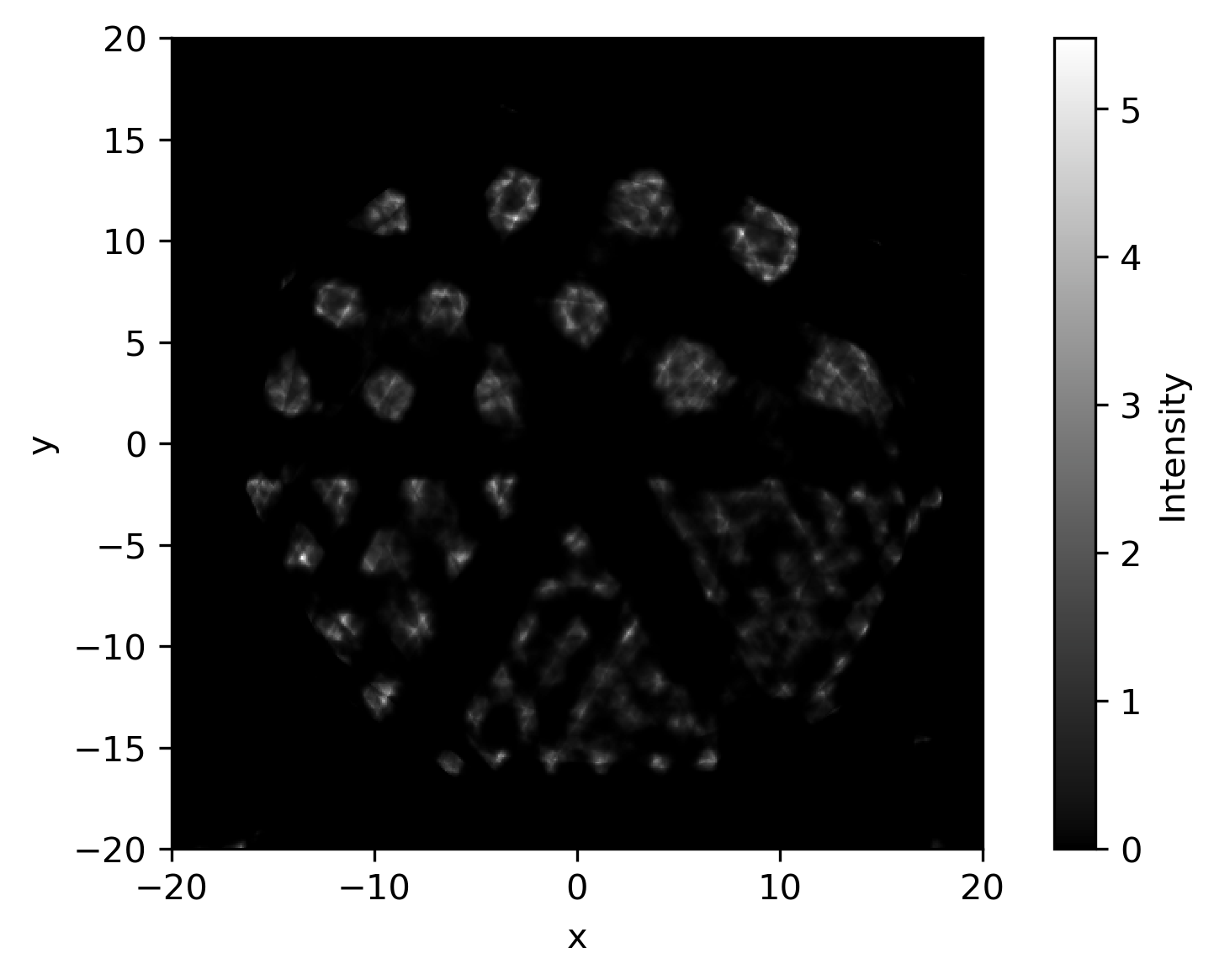} 
\caption{} \label{}
\end{subfigure} 
  \caption{Case $t=10^{-1}$. (A) Divergence along iterates. (B) Maximum of reconstruction along iterates. (C) Reconstruction after $100$ iterates. (D) Reconstruction after $1000$ iterates.}
  \label{time-1}
\end{figure}

  \subsection{Examples with \(\beta = 2\) and Regularisation}
\subsubsection{Toy Example}  
We illustrate singularity with a toy example.
There will just be two detectors \(a_1\) and \(a_2\), so \(\ndet = 2\).
We choose specifically
\[ K = [0,1], \qquad a_1 = 1, \quad a_2(x) = x.\]
With such analytic $a_i$ functions, \autoref{analytic_case} applies.

We also choose \(\beta = 2\). Recall from \autoref{sec:betdiv} that the divergence reduces to the Euclidean distance in this case.
We can now compute the singular solutions explicitly depending on the parameter \(y \in \R^{2}\), as shown in \autoref{fig:regions}. For $y \in \R_+^2$ outside of the cone \[\imcone =\{(y_1, y_2) \in \R^2, \, 0 \leq y_1 \leq y_2\},\] the solution is of the form $\xi \delta_1$ with $\xi \geq 0$ a varying parameter.  For completeness, we also display the solution for $y \notin \R_+^2$ (which is either $0$ or $\xi \delta_0$ with $\xi \geq 0$ a varying parameter).

We also look at the effect of regularisation. 
In this case, following standard practice in many image reconstruction problems, we use total variation regularisation, that is, we solve 
\[ \min_{\mu \geq 0}  \quad \loss(\mu) +\rho \operatorname{TV}(\mu),\] where \(\operatorname{TV}(\mu)\) is the total variation of the derivative of the measure \(\mu\) and \(\rho>0\) is a regularisation parameter.
In this discretised, one-dimensional setting, this is simply \(\operatorname{TV}(\mu) = \sum_i \abs{\mu^{i+1} - \mu^{i}}\), where \(\mu^i\) is the value of the discretised measure at pixel~\(i\).
We then compute the minimum using a primal-dual hybrid gradient method~\cite{ChaPo10}.
We plot the resulting minima for various values of the regularisation parameter \(\rho\) in~\autoref{fig:toyexperiment}, in the case $y = (0,1)$. As $\rho$ goes to $0$, the solution approaches the expected singular limit $\textstyle \frac{1}{2} \delta_1$.

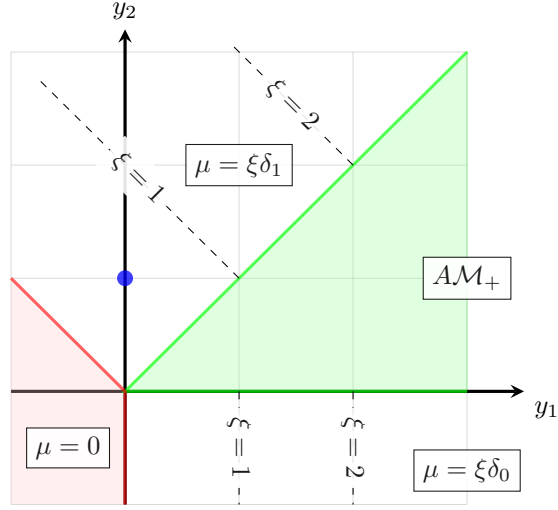
\begin{figure}
  \centering

  \begin{tikzpicture}[
    scale=1.5,
    vector/.style={draw,thick,->,},
    boundary/.style={draw,very thick,opacity=.6},
    zerob/.style={red},
    coneb/.style={green},
    caxis/.style={draw,->,very thick},
    isoval/.style={draw,thin,dashed},
    isolab/.style={sloped, fill=white,opacity=.9},
    reglab/.style={fill=white,opacity=.9,draw=black},
    ]
      \coordinate (O) at (0,0);
      \coordinate (A) at (2,0);
      \coordinate (M) at (2,2);
      \coordinate (P) at (-1,1);
      \coordinate (Q) at (0,-1);
      \coordinate (C) at (-1,-1);

      \draw[step=1cm,very thin, gray,opacity=.3] (C) grid(3,3);

      \draw[caxis] (-1,0) --  (3.5,0) node[below right]{$y_1$};
      \draw[caxis] (0,-1) -- (0,3.2) node[above]{$y_2$};

      \path[fill=green!40!white,draw=green,opacity=.3] (O) -- (3,0) -- (3,3) -- cycle;
      \path[fill=red!20!white,opacity=.3] (O) -- (P) -- (C) -- (Q) -- cycle;
      \path[boundary,coneb] (3,0) -- (O) -- (3,3);
      
      \path[boundary,zerob] (P) -- (O) -- (Q);

      \path[isoval] (1,1) -- node[isolab]{\(\xi=1\)} +(135:2.5);
      \path[isoval] (2,2) -- node[isolab]{\(\xi=2\)} +(135:1.5);

      \path[isoval] (0:1) -- node[isolab]{\(\xi=1\)} +(-90:1);
      \path[isoval] (0:2) -- node[isolab]{\(\xi=2\)} +(-90:1);

      \node[reglab] at (3,-.7) {\(\mu=\xi \delta_0\)};
      \node[reglab] at (1,2) {\(\mu=\xi \delta_1\)};
      \node[reglab] at (-0.5,-.5) {\(\mu=0\)};
      \node[reglab] at (3,1) {\(A\mathcal{M}_+\)};

      \fill[blue,opacity=.75]  (0,1) circle (2pt);

    \end{tikzpicture}
    \caption{
      The singular solutions of the problem \eqref{opt} with two detectors, \(a_1 = 1\), \(a_2(x)=x\) on the interval \(K=[0,1]\).
      There are three distinct regions outside the cone \(A \MM_{+}\), but only one which intersects the first quadrant $y \in \R_+^2$, where the optimal solution is a Dirac, given by \(\xi\delta_{1}\) for some $\xi \geq 0$.
    }
 \label{fig:regions} 
\end{figure}

\begin{figure}
  \centering
  \includegraphics[width=.8\textwidth]{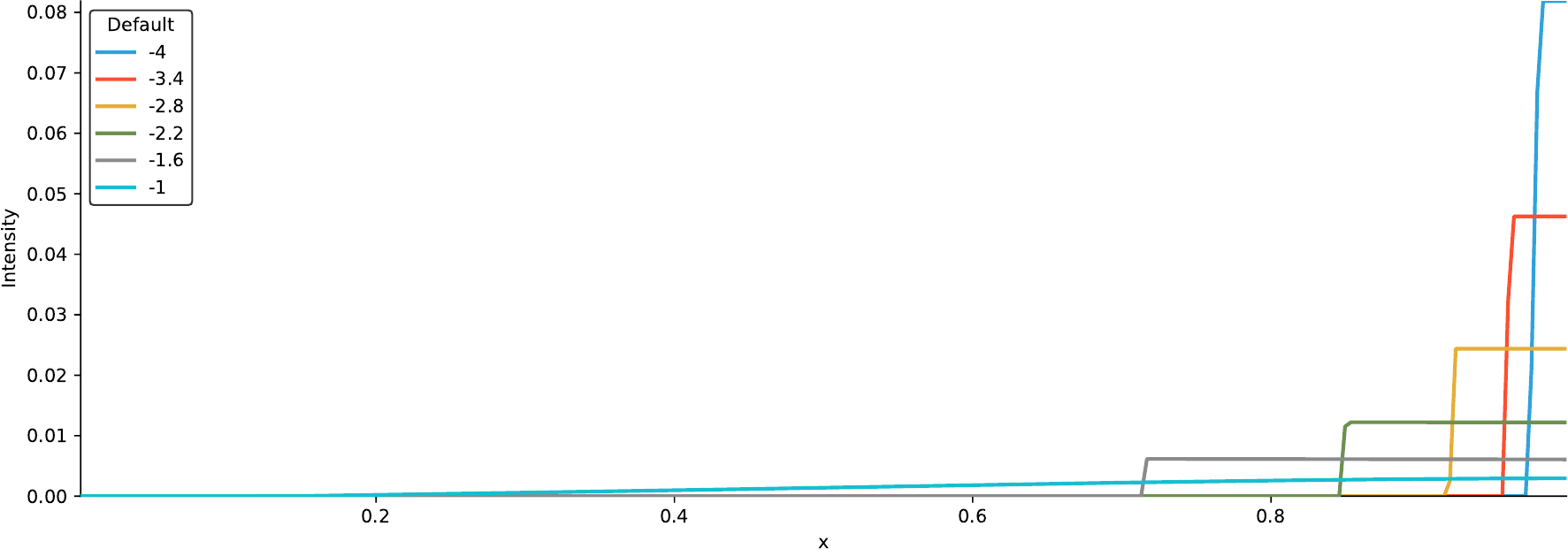}
  \caption{
    The solutions for $y = (0,1)$ of the problem in \autoref{fig:regions} with total variation regularisation.
    Each curve corresponds to a different regularisation parameter, labelled by its base-10 logarithm.
    We see that when the regularisation parameter goes to zero, the computed solution converges to the expected exact solution, which,
    we see from \autoref{fig:regions},
    is \(\mu = .5 \delta_1\), depicted by a blue circle.
  }
  \label{fig:toyexperiment}
\end{figure}

\subsubsection{Tomography example}

We finally look at a more realistic 2D example taken from tomography, where the unknown $\mu$ equals the usual Shepp-Logan phantom used as a benchmark in CT tomography. The example of~\autoref{fig:rhomax} features  an image resolution of \(127 \times 127\), and there are \num{285} angles and \num{183} tangential coordinates.

We consider the case of the Euclidean distance $d_2$. The data $y$ is obtained by Gaussian draws with negative values clipped to $0$, {i.e.}, $y \sim \max(\mathcal{N}(A\mu, \sigma^2),0)$, which ensures $y\in \R_+^\ndet$.  We then solve the corresponding TV-regularised problem
\[ \min_{\mu \geq 0}  \quad D_2(y, A \mu) +\rho \operatorname{TV}(\mu),\]
where $\rho>0$, by a primal-dual hybrid gradient method. Here, $\operatorname{TV}$ refers to isotropic~$\operatorname{TV}$.

As~\autoref{fig:rhomax} shows, the maximum of the reconstruction gets bigger as $\rho$ tends to~$0$. In fact, the reconstruction for small $\rho$ clearly exhibits the night-sky effect. The fact that $\rho = 0$ does not lead to higher maximum values is due to the resolution, which acts as a regulariser for the expected singular measures.

\subsection*{Acknowledgments.}
The authors are indebted to an anonymous referee whose remarks greatly improved the first versions of this paper.

\begin{figure}
  \centering
  \begin{subfigure}{.7\textwidth}
    \begin{tikzpicture}
      \begin{axis}[
        xlabel={Regularisation parameter \(\rho\)},
        ylabel={Maximum over the reconstruction},
        grid=major,
        xmin=0,
        xmax=6,
        ymin=0,
        ]
        \addplot[blue, mark=*, line width=2pt, mark size=1pt] table {Figures/regmaxdata.dat};
      \end{axis}
    \end{tikzpicture}
    \caption{The maximum of the reconstruction for various regularisation parameters \(\rho\)}
  \end{subfigure} \\
    \begin{subfigure}{.45\textwidth}
    \includegraphics[width=\linewidth]{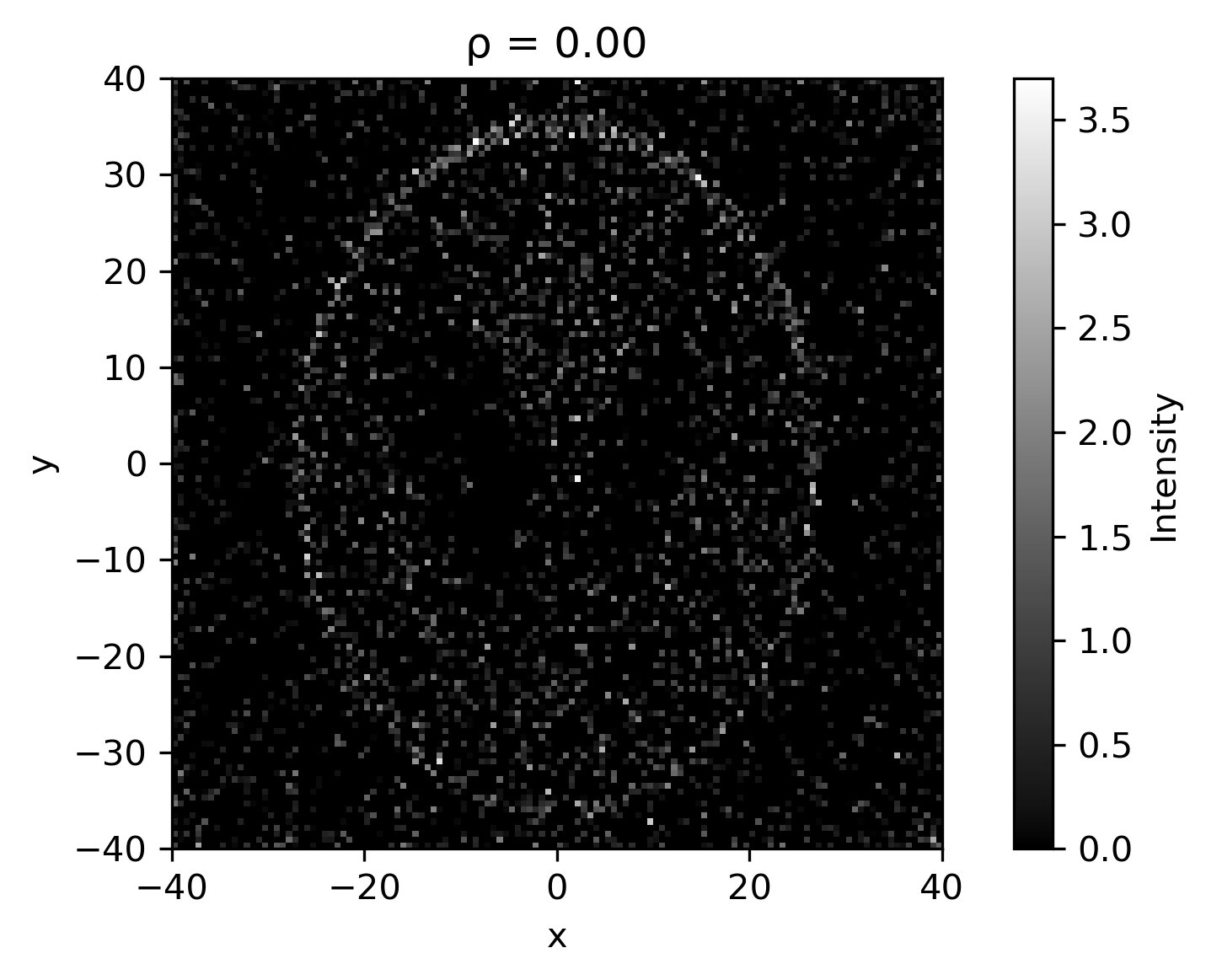}
    \caption{Reconstruction with no regularisation (\(\rho=0\))}
  \end{subfigure} 
  \begin{subfigure}{.45\textwidth}
    \includegraphics[width=\linewidth]{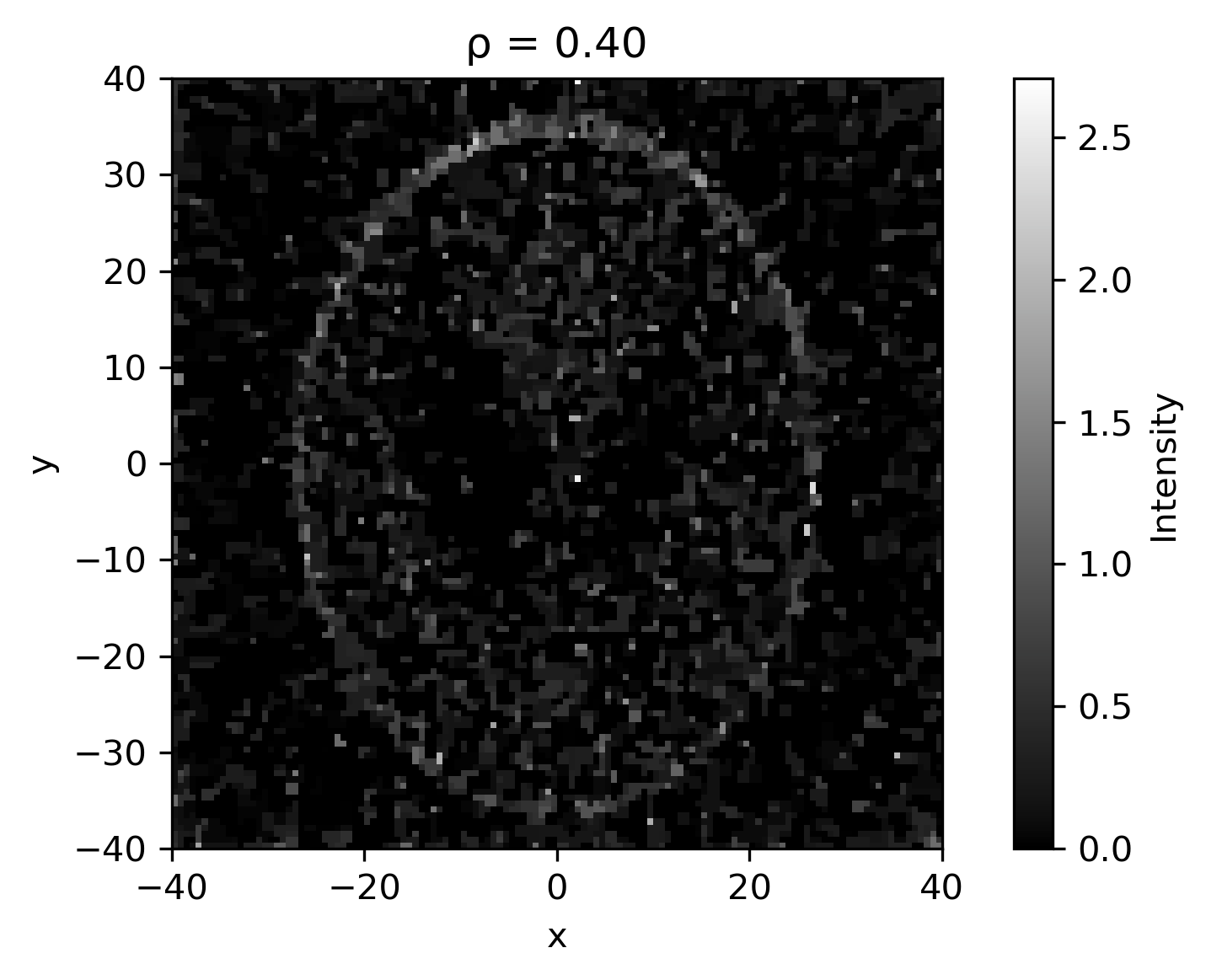}
    \caption{Reconstruction with small regularisation (\(\rho=0.4\))}
  \end{subfigure} \\
  \begin{subfigure}{.45\textwidth}
    \includegraphics[width=\textwidth]{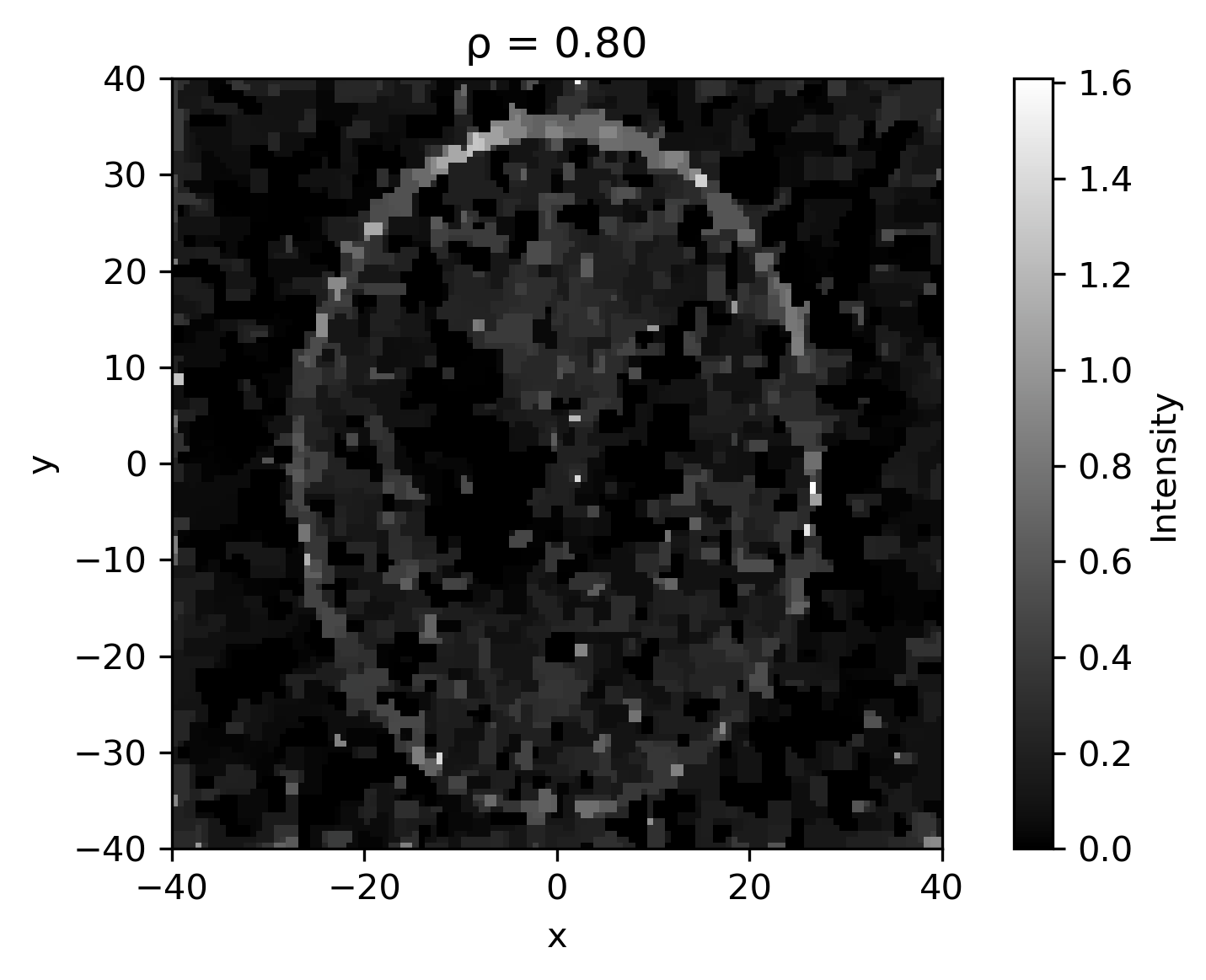}
    \caption{Reconstruction with medium regularisation (\(\rho=0.8\))}
  \end{subfigure}
  \begin{subfigure}{.45\textwidth}
    \includegraphics[width=\textwidth]{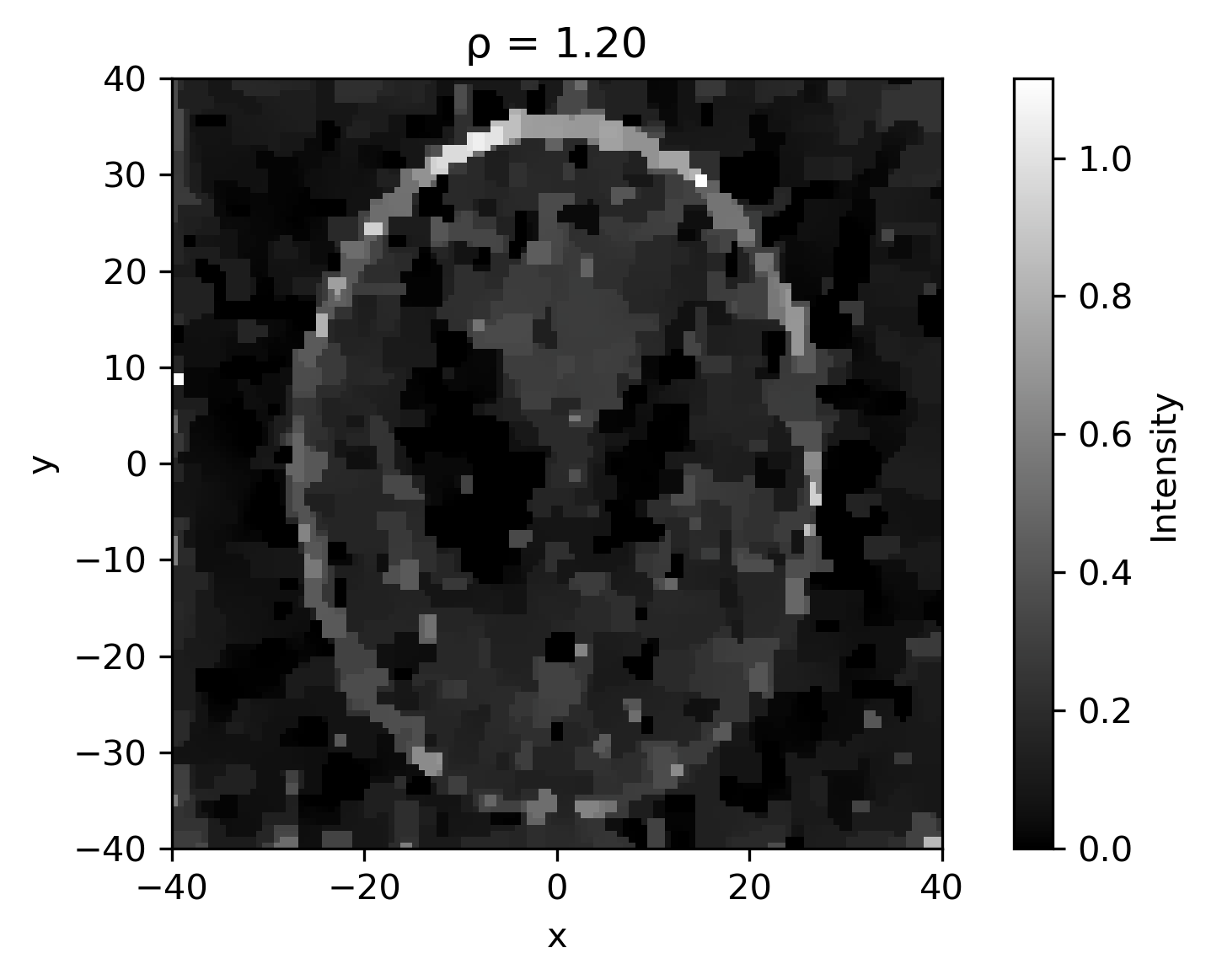}
    \caption{Reconstruction with high regularisation (\(\rho=1.2\))}
  \end{subfigure}
  \caption{Relation between regularisation parameter and the reconstruction of the Shepp--Logan phantom $\mu$. Here, the data is given by $y\sim\max(\mathcal{N}(A\mu, \sigma^2),0)$, with $\sigma = 10$.}
  \label{fig:rhomax}
\end{figure}

\bibliography{biblio.bib}
\bibliographystyle{acm}

\appendix
\section{Computation of the dual}
\label{appA}
Let us denote \[\psi_y (w) := d_\beta(y,w)- \lambda w.\]

The goal is to compute the function $h$ defined by~\eqref{eq:defh}.

 \subsection{Case $ \beta = 1$.}
We have $\psi_y(w) \sim (1- \lambda) w$ as $w \to +\infty$, hence $h(y, \lambda) = - \infty$ if $\lambda >1$.
If $\lambda =  1$ and $y >0$, $\psi_y(w) \sim -y \ln(w)$ as $w \to 0$: the function tends to $-\infty$ as $w \to 0$ and $h(y, \lambda) = - \infty$. If $y = 0$, the function equals $0$ identically and its minimum is $0$.

We now focus on the case $\lambda < 1$. 
We still have $\psi_y(w) \sim (1- \lambda) w$ as $w \rightarrow +\infty$. 
If $y>0$, $\psi_y(w) \sim - y \ln(w)$ as $w \rightarrow 0$, thus the function tends to $+\infty$ at both ends.
Since $\psi_y$ is strictly convex in this case, it has a unique minimum for $w>0$, which we again denote $w(y, \lambda)$, solving
\[w(y, \lambda)^{-1}(w(y, \lambda) - y) = \lambda \iff w(y, \lambda) = \frac{y}{1- \lambda}. \]
If $y =0$, the function $\psi_y(w) = w - \lambda w$ is minimised at $w = 0$, with value $0$.


We may also gather the cases $y = 0$ and  $y>0$ whenever $\lambda \leq 1$, since the formula  for $w(y, \lambda)$ shows that it vanishes with $y$.

Summing up, we find 
\[
  h(y, \lambda) = \begin{cases}
- \infty & \text{ if  \(\lambda > 1\)}
\\
d_1(y, w(y, \lambda)) - \lambda w(y, \lambda) & \text{ if \(\lambda \leq 1\)}
\end{cases}
\]
 In the last case, further computations lead to 
 \[d_1(y,w(y, \lambda)) - \lambda w(y, \lambda) = y \ln(1-\lambda).\]

 \subsection{Case $ 1 < \beta < 2$.}
If $y>0$, $\psi_y(w) \sim \frac{w^\beta}{\beta}$ as $w \rightarrow +\infty$, and $\psi_y(w) \sim - \frac{y}{\beta-1} w^{\beta - 1}$ as $w \rightarrow 0$. The derivative $\psi'_y$ satisfies  $\psi'_y(w) \sim -y w^{\beta - 2}$ as $w \rightarrow 0$, and $\psi'_y(w) \sim w^{\beta - 1}$ as $w \rightarrow +\infty$. Thus the function $\psi'$ increases (as $\psi$ is convex) from $-\infty$ to $+\infty$. As a consequence, it has a unique minimum $w>0$, which we again denote $w(y, \lambda)$, solving
\[w(y, \lambda)^{\beta - 2}(w(y, \lambda) - y) = \lambda.\]
If $y =0$ and $\lambda \leq 0$, it is easily seen that the function $\psi_y(w) = \frac{1}{\beta}w^\beta - \lambda w$ is minimised at $w = 0$, with value $0$, whereas if $\lambda > 0$, it has a unique minimum (also defined by the equation for $w(y,\lambda)$). 
Summing up, we find 
\[
  h(y, \lambda) = \begin{cases}
0 & \text{ if \(y = 0, \lambda \leq 0\)}
\\
d_\beta(y,w(y, \lambda)) - \lambda w(y, \lambda) & \text{otherwise}. 
\end{cases}
\]

  \subsection{Case $\beta = 2$.}
In this case, $\psi_y$ has a unique minimum $w$ given by $w = \lambda + y$, which gives the explicit formula
 \[ h(y, \lambda) = - \frac{1}{2}(\lambda  +y)^2 + \frac{1}{2} y^2.\]

\end{document}